%% file: lqc.tex
\documentclass[11pt]{amsart}

\usepackage[T1]{fontenc}
\usepackage[utf8]{inputenc}
\usepackage{lmodern}

\usepackage{layout}
\usepackage[american]{babel}
\usepackage[babel, final]{microtype}
\usepackage{amssymb}
\usepackage{ulem}
\usepackage[inline]{enumitem}

\usepackage{xifthen}

\usepackage[pdftex,%
  dvipsnames,%
]{xcolor}
\usepackage[pdftex, final]{graphicx}
\usepackage{pinlabel}
\usepackage[pdftex,%
  letterpaper,%
  includehead,%
  includefoot,%
  nomarginpar,%
  lmargin=1.7in,%
  rmargin=1.7in,%
  tmargin=1.5in,%
  bmargin=1.5in,%
]{geometry}
\usepackage[pdftex,%
  final,%
  colorlinks=true,%
  linkcolor=NavyBlue,%
  citecolor=NavyBlue,%
  filecolor=NavyBlue,%
  menucolor=NavyBlue,%
  urlcolor=NavyBlue,%
  bookmarks=true,%
  bookmarksdepth=3,%
  bookmarksnumbered=true,%
  bookmarksopen=true,%
  bookmarksopenlevel=2,%
]{hyperref}
\hypersetup{
  pdfinfo={
    Title={Lagrangian zigzag cobordisms},
    Author={Joshua M. Sabloff, David Shea Vela-Vick, C.-M. Michael Wong, Angela Wu},
    Subject={Contact geometry, knot theory, Lagrangian cobordism},
    Keywords={Legendrian links, Lagrangian cobordisms}
  }
}

\usepackage{tikz-cd}

\input{macros}

\input{commands}

\input{commands_local}

\title{Lagrangian Zigzag Cobordisms}

\author[J.\ M.\ Sabloff]{Joshua M. Sabloff}
\address{Department of Mathematics and Statistics \\ Haverford College \\  
  Haverford, PA 19041}
\email{\href{mailto:jsabloff@haverford.edu}{jsabloff@haverford.edu}}
\urladdr{\url{https://jsabloff.sites.haverford.edu}}

\author[D.\ S.\ Vela-Vick]{David Shea Vela-Vick}
\address{Department of Mathematics \\ Louisiana State University \\ Baton 
  Rouge, LA 70803}
\email{\href{mailto:shea@math.lsu.edu}{shea@math.lsu.edu}}
\urladdr{\url{https://www.math.lsu.edu/~shea/}}

\author[C.-M.\ M.\ Wong]{C.-M. Michael Wong}
\address{Department of Mathematics and Statistics \\ University of Ottawa \\  
  Ottawa, ON K2P 1C8}
\email{\href{mailto:Mike.Wong@uOttawa.ca}{Mike.Wong@uOttawa.ca}}
\urladdr{\url{https://mysite.science.uottawa.ca/cwong/}}

\author[A.\ Wu]{Angela Wu}
\address{Department of Mathematics \\ Louisiana State University \\ Baton 
  Rouge, LA 70803}
\email{\href{mailto:awu@lsu.edu}{awu@lsu.edu}}
\urladdr{\url{https://www.math.lsu.edu/~awu/}}


\begin{document}

\normalem

\begin{abstract}
We investigate an equivalence relation on Legendrian knots in the standard contact
three-space defined by the existence of an interpolating zigzag of Lagrangian 
cobordisms. We compare this relation, restricted to genus-$0$ surfaces, to smooth 
concordance and Lagrangian concordance. We then study the metric monoid formed 
by the set of Lagrangian zigzag concordance classes, which parallels the metric 
group formed by the set of smooth concordance classes, proving structural 
results on torsion and satellite operators.  Finally, we discuss the relation 
of Lagrangian zigzag cobordism to non-classical invariants of Legendrian knots.
\end{abstract}

\maketitle



\input{sec_intro}
\input{sec_lqc}
\input{sec_satellite}

\input{sec_maslov}
\input{sec_lqc-e}


\bibliographystyle{mwamsalphack}
\bibliography{bibliography}


\end{document}

%% file: commands.tex
\newcommand{\set}[1]{\mathchoice%
  {\left\lbrace #1 \right\rbrace}%
  {\lbrace #1 \rbrace}%
  {\lbrace #1 \rbrace}%
  {\lbrace #1 \rbrace}%
}

\newcommand{\paren}[1]{\mathchoice%
  {\left( #1 \right)}%
  {( #1 )}%
  {( #1 )}%
  {( #1 )}%
}

\newcommand{\abs}[1]{\mathchoice%
  {\left\lvert #1 \right\rvert}%
  {\lvert #1 \rvert}%
  {\lvert #1 \rvert}%
  {\lvert #1 \rvert}%
}

\newcommand{\union}{\cup}

\newcommand{\cross}{\times}

\newcommand{\numset}[1]{\mathbb{#1}}
\newcommand{\N}{\numset{N}}
\newcommand{\Z}{\numset{Z}}

\newcommand{\R}{\numset{R}}

\newcommand{\bdy}{\partial}
\newcommand{\connsum}{\mathbin{\sharp}}

\DeclareMathOperator{\tb}{tb}
\DeclareMathOperator{\rot}{r}

%% file: commands_local.tex
\newcommand*{\dfn}[1]{\textbf{#1}}

\newcommand{\disjunion}{\sqcup}

\newcommand{\FF}{\ensuremath{\mathbb{F}}}

\newcommand*{\leg}{\Lambda}

\DeclareMathOperator{\surg}{Surg}

\newcommand*{\aug}{\varepsilon}
\newcommand*{\poly}{\ensuremath{\mathcal{P}}}
\newcommand*{\alg}{\ensuremath{\mathcal{A}}}

\newcommand*{\qconc}{\approx}
\newcommand*{\qcob}{\sim}
\newcommand*{\qconce}{\stackrel{\aug}{\approx}}
\newcommand*{\qcobe}{\stackrel{\aug}{\sim}}
\newcommand*{\qconcd}{\qconc_{\mathrm{d}}}

\newcommand*{\legjoin}{\mathbf{L}}
\newcommand*{\qcobordism}{\mathbb{L}}

\newcommand*{\graph}{\Gamma}
\newcommand*{\handles}{\mathcal{H}}
\newcommand*{\vertices}{\mathcal{V}}
\newcommand*{\cusps}{\mathcal{C}}

\newcommand*{\maxunknot}{\Upsilon}

\newcommand*{\gd}{\mathbb{G}}

\newcommand*{\LagGenus}{g_{\mathrm{L}}}

\newcommand*{\ConcGrp}{\mathcal{C}}

\newcommand*{\LagCobRel}{\prec}

\newcommand*{\LZC}{\mathcal{LZC}}

\DeclareMathOperator{\LCH}{LCH}

\newcommand*{\concat}{\mathbin{\odot}}
\DeclareMathOperator{\Wh}{Wh}

\newcommand*{\taxicab}{d_{\mathrm{t}}}

\newcommand*{\loss}{\mathcal{L}}
\newcommand*{\lossh}{\widehat{\loss}}

\DeclareMathOperator{\HFK}{HFK}
\newcommand*{\HFKh}{\widehat{\HFK}}

\newcommand*{\arcindex}{\alpha}

%% file: sec_intro.tex



\section{Introduction}
\label{sec:intro}

\subsection{Context and aims}
\label{ssec:context-aims}

An essential framework in the study of contact and symplectic topology is to 
explore the boundary between flexibility and rigidity.  A prominent locus for 
this study involves questions about the existence, (non-)uniqueness, and 
structure of exact, orientable Lagrangian concordances between Legendrian 
knots.  Evidence for rigidity arises from the development of obstructions 
beyond the topological to the existence of such cobordisms 
\cite{BalLidWon22:LagCobHFK, BalSiv18:KHMLeg, EkhHonKal16:LagCob, 
  GolJuh19:LOSSConc, Pan17:LagCobAug, SabTra13:LagCobObstructions} and the 
discovery of many symplectically distinct but smoothly isotopic Lagrangian 
concordances \cite{CasGao22:InfLag, CasNg22:InfLag, Kal05:LCHpi1}.  Of 
particular importance for this paper is the structural fact that the relation on the set of Legendrian knots
defined by Lagrangian concordance is reflexive and transitive, but not 
symmetric \cite{Cha10:LagConc, Cha15:LagConcNotSym}, though whether Lagrangian 
concordance yields a partial order is still an open question.  This suggests 
that Lagrangian concordance is more akin to smooth ribbon concordance, which 
was recently shown by Agol to induce a partial order 
\cite{Ago22:RibConcPartialOrder}, than to ordinary smooth concordance, which 
yields an equivalence relation.

The goal of this paper is to introduce a new relation we term \dfn{Lagrangian 
  zigzag concordance} which lies between smooth and Lagrangian concordance, and 
to explore its structure and whether it still captures symplectic rigidity.  
Following \cite{SabVelWon21:MaxLeg}, we say that Legendrians $\leg$ and $\leg'$ 
are \dfn{Lagrangian zigzag concordant} if they are related by a sequence of 
connected Lagrangian concordances as in the following zigzag diagram:
\begin{equation} \label{eq:zz}
\begin{tikzcd}[column sep=small, row sep=small]
  & \leg_{+,1} && && \leg_{+,n} & \\
	\leg = \leg_0 \ar[ur,"L^<_1"] && \leg_1 \ar[ul, "L_1^>"'] & \cdots & \leg_{n-1} \ar[ur,"L^<_n"] && \leg_n  = \leg' \ar[ul,"L^>_n"']
\end{tikzcd}
\end{equation}	See \fullref{sec:lzc} for a precise definition.  We note that 
Sarkar has defined a similar notion for ribbon concordance \cite{Sar20:RibKh}.

The notions of zigzag cobordism and concordance allow us to pursue a closer 
analogy between contact and smooth knot theory.  To begin, the Lagrangian 
zigzag concordance relation is an equivalence relation that refines the smooth 
concordance relation. Similarly to smooth knots, any two Legendrians with the 
same rotation number are related by zigzag cobordism \cite{SabVelWon21:MaxLeg}.  
Thus, we may define a metric on the set $\LZC$ of zigzag concordance classes of 
Legendrians given by the minimum total genus $\LagGenus$ over all Lagrangian 
zigzag cobordisms (setting the distance between two Legendrians to be $\infty$ 
if they have different rotation numbers), which parallels the smooth 
concordance genus $g_4$ on the smooth concordance group $\ConcGrp$.

The motivating questions for this paper may now come into sharper focus:  
First, how does the zigzag concordance relation refine the smooth concordance 
relation, and how do their structures compare?  We can ask about both the 
metric structure and the algebraic structure of the relation. For example, what 
is the diameter of a component of $\LZC$ that corresponds to a fixed rotation 
number? Can we find interesting quasi-isometries of $\LZC$? Does connected sum 
induce a group structure on the set of zigzag concordance classes?  Second, 
does the additional flexibility of the Lagrangian zigzag concordance relation 
still allow for symplectic rigidity?  In particular, how does the zigzag 
concordance relation interact with existing non-classical invariants?  We 
enumerate our efforts to answer these questions in the results detailed below.

While many (but not necessarily all) statements in the article hold for general 
$(M, \alpha)$ with essentially the same proof, for ease of exposition, we focus 
on $(\R^3, \mathrm{d}z - y \mathrm{d}x)$ throughout.

\subsection{Comparisons}
\label{ssec:comparisons}

We first study the differences between smooth concordance, Lagrangian 
zigzag concordance, and Lagrangian concordance. It has already been shown that 
smooth and Lagrangian zigzag concordance do not coincide 
\cite[Example~6.6]{SabVelWon21:MaxLeg}. Just how similar these notions are, 
however, can be seen if we drop the assumption that the individual cobordisms 
in \eqref{eq:zz} are connected (even if the total cobordism is still required 
to be a cylinder).  We call such an object a \dfn{disconnected Lagrangian 
  zigzag-concordance}. It turns out that this notion coincides with smooth 
concordance:

\begin{theorem}
  \label{thm:disconnected-lqc}
  The Legendrian knots $\leg$ and $\leg'$ are smoothly concordant if and only 
  if they are disconnected Lagrangian zigzag concordant.
\end{theorem}

On the other side, we show that Lagrangian zigzag concordance is coarser than Lagrangian 
concordance:

\begin{theorem}
  \label{thm:lqc-not-c}
  There exist a pair of Legendrian knots $\leg$ and $\leg'$ that are Lagrangian 
  zigzag concordant but not Lagrangian concordant.
\end{theorem}

\subsection{Structural results}
\label{ssec:structural}

We next consider the metric properties of $\LZC$.  Thinking of $\LZC$ as a 
graph weighted by $\LagGenus$, we begin by investigating its basic properties:

\begin{proposition} \label{prop:graph-structure}
  The graph $\LZC$ has exactly one connected component for each rotation 
  number. Each component has infinite diameter.  The $1$-link of each vertex of 
  $\LZC$ is countably infinite.
\end{proposition}

We then study $\LZC$ as a monoid under connected sum, both proving results and 
raising questions. For example, in contrast to the smooth case, it is unknown 
whether any non-trivial Lagrangian zigzag concordance class has an inverse.  We 
open the topic by investigating the existence of torsion, proving a result that 
contrasts the case of $\ConcGrp$:

\begin{theorem}
\label{thm:amph-not-torsion}
  No non-trivial amphicheiral Legendrian knot is $2$-torsion in $\LZC$.
\end{theorem}

Paralleling work of Cochran and Harvey \cite[Proposition~4.1]{CocHar18:ConcGeom}, we show that $\LZC$ is not $\delta$-hyperbolic. Further, inspired by \cite[Theorem~6.5]{CocHar18:ConcGeom}, 
which shows that every winding-number-$\pm 1$ satellite operator is a quasi-isometry of $\ConcGrp$ to 
itself, we prove that $\LZC$ has a similar property:

\begin{theorem} \label{thm:winding-1-qis}
  Any winding-number-$\pm 1$ satellite operator with rotation number $0$ 
  relative to the identity or reverse operator is a quasi-isometry of $\LZC$ to 
  itself.%
\end{theorem}

The proof of this statement requires extending  
\cite[Proposition~3.1]{SabVelWon21:MaxLeg}, which uses convex surface theory to 
obtain a span (i.e.\ a pair from a common lower bound) of decomposable 
Lagrangian cobordisms between any two null-homologous Legendrian knots in any 
contact $3$-manifold.  We establish the analogous statement 
(\fullref{lem:lqc-satellite-general}) for homologous knots in a possibly 
non-trivial homology class.

\subsection{Relation to non-classical invariants}
\label{ssec:relation-non-classical}

To better understand how non-classical invariants of Legendrian knots interact 
with zigzag concordance, we specialize to the Legendrian contact homology 
differental graded algebra (LCH DGA).  The effective use of LCH often requires 
Lagrangian cobordisms to have Maslov number $0$. Therefore, we prove a 
different strengthening of \cite{SabVelWon21:MaxLeg} to produce Maslov-$0$ 
cospans for Legendrian links in $\R^3$:

\begin{theorem}[cf.\ {\cite[Theorem~1.1]{SabVelWon21:MaxLeg}}]
  \label{thm:maslov-0}
  If $\leg$ and $\leg'$ are oriented Legendrian links in the standard contact 
  $\R^3$, all of whose components have vanishing rotation number, then there 
  exist a Legendrian link $\leg_+$ and Maslov-$0$ Lagrangian cobordisms from $\leg$ to $\leg_+$ and from $\leg'$ to $\leg_+$.
\end{theorem}

As we shall see in \fullref{ssec:lch-zz}, however, the structure of the known 
non-classical invariants makes it difficult to use them to obstruct zigzag 
concordance.  On one hand, this serves as motivation to develop a more refined 
understanding of these invariants under Lagrangian cobordism or to develop new 
styles of non-classical invariants.  On the other, the precise nature of the 
failure of linearized LCH to provide an effective obstruction to zigzag 
concordance yields an interesting result, as follows.  The set of linearized 
LCH, taken over all possible augmentations of the LCH DGA, is usually encoded 
by the set $\mathcal{P}_\leg$ of corresponding Poincar\'e--Chekanov polynomials.  Zigzag cobordism yields 
a solution to the setwise geography problem for Poincar\'e--Chekanov 
polynomials, culminating in the following theorem:

\begin{theorem} \label{thm:poly-geography}
  For any finite set $\{ p_1(t), \ldots, p_n(t)\}$ of polynomials with 
  non-negative integral coefficients, there exist a Legendrian knot $\leg$ and 
  non-negative even integers $c_1, \ldots, c_n$ so that \[\{p_1(t) + 
    p_1(t^{-1}) +t + c_1, \ldots, p_n(t) + p_n(t^{-1}) + t+c_n\} \subset 
    \poly_\leg.\]
\end{theorem}

This theorem generalizes a theorem of Melvin and Shrestha 
\cite{MelShr05:ChePolys} that any single polynomial of the form $p(t) = q(t) + 
q(t^{-1}) +t$ can be realized as the linearized Legendrian contact homology of 
a knot with respect to an augmentation; see \cite{BouGal22:GeoBilinLCH, 
  BouSabTra15:LagCobGF} for related results.

\subsection{Organization}

Throughout the paper, we assume familiarity with fundamental notions of 
Legendrian knots and contact topology as in \cite{Etn05:LegTransSurvey, 
  Gei08:ContactBook, Tra21:LegTransIntro}.

The article is organized as follows: We begin by defining Lagrangian zigzag 
cobordism in \fullref{sec:lzc}, and then comparing and contrasting that 
definition with smooth and Lagrangian cobordism.  We then proceed to 
investigate the structure of the Lagrangian zigzag cobordism graph in 
\fullref{sec:zz-graph}, including proofs of \fullref{prop:graph-structure}, 
\fullref{thm:amph-not-torsion}, and \fullref{thm:winding-1-qis}.  Finally, we 
discuss interactions between zigzag cobordisms and non-classical invariants, 
first investigating the existence of Maslov-$0$ zigzag cobordisms in 
\fullref{sec:maslov-0}, then treating Legendrian contact homology in 
\fullref{sec:non-classical}.

\subsection*{Acknowledgments}

The authors thank Dan Rutherford for raising the question that led to 
\fullref{thm:poly-geography} and Allison N.\ Miller for suggesting the examples 
in the proof of \fullref{prop:valence}.  Part of the research was done while 
JMS was hosted by the Institute for Advanced Study; in particular, this paper 
is partly based on work supported by the Institute for Advanced Study.  DSV was 
partially supported by NSF Grant DMS-1907654. CMMW was partially supported by 
NSF Grant DMS-2238131 and an NSERC Discovery Grant.  Part of the research was 
done while CMMW was at Louisiana State University, and he thanks the department 
for its support. AW was partially supported by an AMS--Simons Travel Grant and 
NSF Grant DMS-2238131. 

%% file: sec_lqc.tex


\section{The Lagrangian Zigzag Cobordism Relation}
\label{sec:lzc}

In this section, we formulate precise definitions of the Lagrangian zigzag 
cobordism and concordance relations on Legendrian links.  We then compare the 
zigzag concordance relation to smooth concordance (proving 
\fullref{thm:disconnected-lqc}) and Lagrangian concordance (proving 
\fullref{thm:lqc-not-c}).

\subsection{Definition of Lagrangian zigzag cobordism}
\label{ssec:lzc-defn}

In order to define Lagrangian zigzag cobordism, we first recall the definition 
of a Lagrangian cobordism.

\begin{definition}
  Let $\leg_-$ and $\leg_+$ be Legendrian links in a contact $3$-manifold $(M, 
  \ker \alpha)$. An \emph{exact Lagrangian cobordism} is an exact embedded 
  orientable Lagrangian surface $L\subset (\R \times M, d(e^t\alpha))$ such 
  that there exists $T$, satisfying that:
  \begin{itemize}
    \item $L\cap ((-\infty, -T)\times M) = (-\infty, -T) \times\leg_-$, denoted 
      $\mathcal{E}_- (L)$;
    \item $L\cap ((T, \infty)\times M) = (T,\infty)\times \leg_+$, denoted 
      $\mathcal{E}_+ (L)$;
    \item There is a smooth function $f \colon L\to \R$ so that $df = 
      e^t\alpha\rvert_L$ and $f$ is constant at each cylindrical end 
      $\mathcal{E}_\pm (L)$; and
    \item $L \setminus (\mathcal{E}_- (L) \union \mathcal{E}_+ (L))$ is compact 
      with boundary $(\set{-T} \times \leg_-) \union (\set{T} \times \leg_+)$.
  \end{itemize}
  In this case we say that $\leg_- \LagCobRel \leg_+$ via the Lagrangian $L$ or 
  that $L$ is a Lagrangian cobordism from $\leg_-$ to $\leg_+$. A Lagrangian 
  cobordism $L$ between Legendrian knots is a \dfn{Lagrangian concordance} if 
  it has genus 0.
\end{definition}

Henceforth, unless specified, we will focus on the case $M = \R^3$ and $\alpha 
= \mathrm{d}z - y \mathrm{d}x$.
 
Classical invariants of Legendrian knots act as obstructions to Lagrangian 
cobordism.  In particular, Chantraine \cite{Cha10:LagConc} proved:

\begin{proposition}[{\cite[Theorem~1.2]{Cha10:LagConc}}] 
  \label{prop:cob-classical}
  If $\leg_- \LagCobRel \leg_+$ via the Lagrangian $L$, then 
  $\rot(\leg_-)=\rot(\leg_+)$ and $\tb(\leg_+)-\tb(\leg_-) = 
  - \chi(L)$.%
  \footnote{The statement in fact holds for null-homologous knots $\leg_\pm$ in 
    a general $(M, \alpha)$, with the rotation numbers taken with corresponding 
    Seifert surfaces.}
\end{proposition}

Lagrangian cobordisms may be constructed using traces of Legendrian isotopies 
or the attachment of Lagrangian $0$- or $1$-handles as in 
\fullref{fig:decomposable_moves} \cite{BouSabTra15:LagCobGF, 
  EkhHonKal16:LagCob}.  A Lagrangian cobordism constructed using these three 
operations is called \dfn{decomposable}.

\begin{figure}[!htb]
  \labellist
  \small\hair 2pt
  \pinlabel {$\leg$} [b] at 28 22
  \pinlabel {$\leg$} [b] at 28 120
  \endlabellist
  \includegraphics{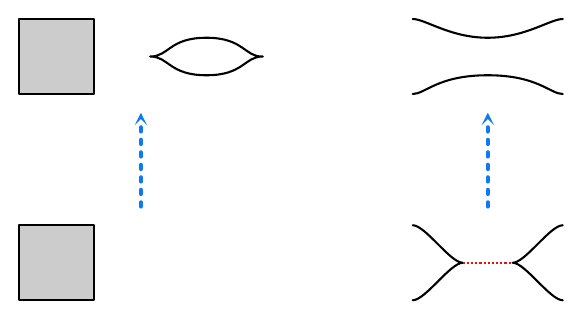}
  \caption{Attaching a Lagrangian $0$-handle (left) and a $1$-handle (right) to 
    a Legendrian link.}
  \label{fig:decomposable_moves}
\end{figure}

To begin the precise definition of Lagrangian zigzag cobordism, let $\leg$ and 
$\leg'$ be two Legendrian links.  A 
\dfn{Lagrangian cospan} $\legjoin(\leg,\leg')$ consists of another Legendrian 
link $\leg_+$ and connected Lagrangian cobordisms $L$ from $\leg$ to $\leg_+$ 
and $L'$ from $\leg'$ to $\leg_+$, which we depict using the following diagram:

\begin{equation*}
  \begin{tikzcd}[column sep=small, row sep = small]
    & \leg_+ & \\
    \leg  \ar[ur,"L"] && \leg' \ar[ul, "L'"']
  \end{tikzcd}
\end{equation*}	

\begin{definition} \label{defn:lqc-relation}
  Two Legendrian links $\leg$ and $\leg'$ are \dfn{Lagrangian zigzag 
    cobordant}, denoted $\leg \qcob \leg'$, if there exist Legendrian links \[
    \leg = \leg_0, \leg_1, \dotsc, \leg_{n-1}, \leg_n=\leg'
  \]
  and Lagrangian cospans $\legjoin(\leg_{i-1},\leg_i)$ for $i=1, \ldots, n$.  
  The data defining a zigzag cobordism is denoted $\qcobordism(\leg, \leg')$.

  If all of the Lagrangian cobordisms in $\qcobordism(\leg, \leg')$ are 
  concordances, then $\leg$ and $\leg'$ are \dfn{Lagrangian zigzag concordant}, 
  denoted $\leg \qconc \leg'$.
\end{definition}
	
See the diagram in \eqref{eq:zz} for a depiction of the data comprised by a 
Lagrangian zigzag cobordism. We note that both $\qcob$ and $\qconc$ are 
equivalence relations on the set of oriented Legendrian links.  
	
\begin{example} \label{ex:zz-from-c}
  A Lagrangian cobordism $L$ from $\leg$ to $\leg'$ induces a Lagrangian zigzag 
  cobordism using the cylindrical cobordism from $\leg'$ to itself:
  \begin{equation*}
    \begin{tikzcd}[column sep=small, row sep = small]
      & \leg' & \\
      \leg  \ar[ur,"L"] && \leg' \ar[ul, "="']
    \end{tikzcd}
  \end{equation*}
\end{example}

\begin{example} \label{ex:zz-not-c}
  Suppose that $\leg$ and $\leg'$ are decomposably Lagrangian slice, i.e.\ 
  there are decomposable Lagrangian concordances $L$ and $L'$ from the maximal 
  unknot $\maxunknot$ to $\leg$ and $\leg'$, respectively.  For example, the 
  Legendrian realizations of the $m (9_{46})$ and $m (12n_{768})$ knots in 
  \fullref{fig:lqc-not-c} are decomposably slice 
  \cite{CorNgSiv16:LagConcObstructions}.
  \begin{figure}[!htb]
    \includegraphics{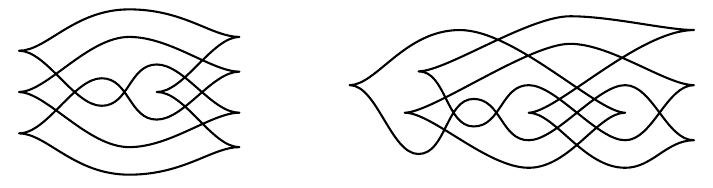}
    \caption{Legendrian realizations of the Lagrangian slice knots $m (9_{46})$ 
      (left) and $m (12n_{768})$ (right).}
    \label{fig:lqc-not-c}
  \end{figure}

  The Legendrians $\leg$ and $\leg'$ are Lagrangian zigzag concordant:
  \begin{equation*}
    \begin{tikzcd}[column sep=small, row sep = small]
      & \leg & & \leg'\\
      \leg  \ar[ur,"="] && \maxunknot \ar[ul,"L"'] \ar[ur,"L'"] && \leg' 
      \ar[ul, "="']
    \end{tikzcd}
  \end{equation*}
\end{example}

We attach several quantities to a Lagrangian zigzag cobordism. The \dfn{Euler 
  characteristic} $\chi(\qcobordism(\leg, \leg'))$ of a Lagrangian zigzag 
cobordism is simply the sum of the Euler characteristics of the constituent 
cobordisms; the genus of a Lagrangian zigzag cobordism is then defined as 
usual.  We use these notions to define a relative genus between two 
Legendrians:

\begin{definition} \label{defn:zz-genus}
  Given two Legendrian knots $\leg$ and $\leg'$, their \dfn{relative Lagrangian 
    zigzag cobordism genus} $\LagGenus(\leg, \leg')$ is the minimal genus of 
  all Lagrangian zigzag cobordisms between $\leg$ and $\leg'$.  
\end{definition}

\begin{remark}
  It is clear that the relative Lagrangian zigzag cobordism genus descends to 
  zigzag concordance classes.
\end{remark}

\begin{remark} \label{rem:zz-genus-neq-smooth-genus}
  In \cite[Section~6]{SabVelWon21:MaxLeg}, it was shown that the smooth 
  relative $4$-genus is a lower bound on the relative Lagrangian zigzag 
  cobordism genus, and that the bound is not always realized.  In particular, 
  the relative Lagrangian zigzag cobordism genus between the maximal unknot and 
  its double stabilization is $1$ while the relative smooth genus is clearly 
  $0$.
\end{remark}

The classical invariants interact nicely with Lagrangian zigzag cobordism and 
its genus.

\begin{proposition} \label{prop:zz-cob-classical}
  Let $\leg$ and $\leg'$ be Legendrian knots. Suppose that $\leg \qcob \leg'$; 
  then $\rot(\leg) = \rot(\leg')$ and \[\left| \tb(\leg)-\tb(\leg')\right| \leq 
    2 \LagGenus(\leg, \leg').\]  In particular, if $\leg \qconc \leg'$, then 
  $\tb(\leg) = \tb(\leg')$.
\end{proposition}

\begin{proof}
  The invariance of the rotation number follows from 
  \fullref{prop:cob-classical}, which implies that all Legendrians in 
  $\qcobordism(\leg, \leg')$ have the same rotation number.  The bound on the 
  difference between Thurston--Bennequin numbers follows from the estimate
  \begin{align*}
    \abs{\tb(\leg)-\tb(\leg')}
    &= \abs{\sum_{i=1}^n \paren{\chi(L^{<}_i) - \chi(L^{>}_i)}}\\
    &\leq \sum_{i=1}^n \paren{-\chi(L^{<}_i) -\chi(L^{>}_i)} \\
    &= -\chi(\qcobordism(\leg,\leg') \\
    &= 2g(\qcobordism(\leg,\leg')).
  \end{align*}
	
  We used the fact that each constituent cobordism has at least two boundary 
  components, and hence has non-positive Euler characteristic; we also used the 
  fact that the Legendrians at the ends are connected in the last line.
\end{proof}

Finally, we define the \dfn{Maslov number} $\mu( \qcobordism(\leg,\leg'))$ of a 
Lagrangian zigzag cobordism to be the greatest common divisor of the Maslov 
numbers of the constituent Lagrangian cobordisms.  Of particular interest are 
zigzag cobordisms with Maslov number $0$.
See \fullref{sec:maslov-0} for further discussion of Maslov-$0$ zigzag 
cobordisms.



\subsection{Comparison to smooth cobordism}
\label{ssec:lqc-vs-smooth}

With the definition and basic properties of Lagrangian zigzag concordance in 
hand, we next explore its relationship with smooth concordance.  In the proof 
of \fullref{prop:zz-cob-classical}, we have seen that the connectedness 
assumption leads to the Thurston-Bennequin number being invariant under 
Lagrangian zigzag concordance, indicating some level of rigidity in the 
relation.  In this section, we will prove \fullref{thm:disconnected-lqc}, which 
shows that without the connectedness condition, the notion of Lagrangian zigzag 
concordance is quite flexible, reducing to its smooth counterpart.

Write $\leg \qconcd \leg'$ if $\leg$ and $\leg'$ satisfy the definition of 
Lagrangian zigzag concordance, but with the intermediate Lagrangians allowed to 
be disconnected even though the total underlying smooth cobordism is still a 
cylinder; in this case, we say that $\leg$ and $\leg'$ are \dfn{disconnected 
  Lagrangian zigzag concordant}.  


The flexibility in disconnected Lagrangian zigzag concordance arises from the 
following construction.

\begin{lemma} \label{lem:double-stab}
  A Legendrian knot $\leg$ is disconnected Lagrangian zigzag concordant to its 
  double stabilization $S_{+-}(\leg)$.
\end{lemma}

\begin{proof}
  The zigzag concordance is constructed in \fullref{fig:double-stab}.
  \begin{figure}[!htb]
    \labellist
    \small\hair 2pt
    \pinlabel {$\leg$} [b] at 220 40
    \pinlabel {$S_{+-}(\leg)$} [b] at 100 20
    \pinlabel {$\leg_+$} [b] at 170 200
    \endlabellist

    \includegraphics{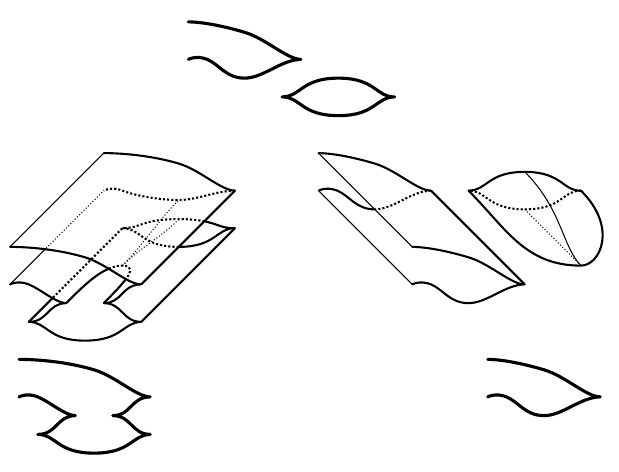}
    \caption{A disconnected Lagrangian zigzag concordance between $\leg$ and 
      $S_{+-}(\leg)$.  The Lagrangian on the left is constructed from a 
      $1$-handle, and the concordance on the right comes from a (cancelling) 
      $0$-handle. The surfaces pictured are front diagrams of the Legendrian 
      lifts of the exact Lagrangian cobordisms.}
    \label{fig:double-stab}
  \end{figure}
\end{proof}

\begin{corollary} \label{cor:disconnected-smooth-isotopy}
  If two Legendrian links $\leg_1$ and $\leg_2$ are smoothly isotopic and the 
  corresponding components have the same rotation numbers, then they are 
  disconnected Lagrangian zigzag concordant.
\end{corollary}

\begin{proof}
  Under the hypotheses of the corollary, Fuchs and Tabachnikov 
  \cite[Theorem~4.4]{FucTab97:LegTransKnots} show that there exist $m, n \in 
  \Z_{\geq 0}$ such that the $m$-fold double stabilization $S_{+-}^m (\leg_1)$ 
  is Legendrian isotopic to $S_{+-}^n (\leg_2)$.  \fullref{lem:double-stab} 
  shows that $\leg_1 \qconcd S_{+-}^m (\leg_1)$ and that $\leg_2 \qconcd 
  S_{+-}^n (\leg_2)$, while the fact that Legendrian isotopy induces a 
  Lagrangian cobordism shows that $S_{+-}^m (\leg_1) \qconcd S_{+-}^n 
  (\leg_2)$. The proof then follows from the transitivity of the Lagrangian 
  zigzag concordance relation.
\end{proof}

\begin{proof}[Proof of \fullref{thm:disconnected-lqc}]
  The reverse direction is clear, so we need only prove that if $F \subset \R^3 
  \times [0,1]$ is a smooth concordance between $\leg$ and $\leg'$, then $\leg$ 
  and $\leg'$ are disconnected Lagrangian zigzag concordant.  Let $r$ denote 
  the common rotation number of $\leg$ and $\leg'$.
	
  The first step is to decompose $F$ into a sequence of elementary cobordisms 
  \[F = F_1 \concat \cdots \concat F_n,\]
  where each $F_i$ is the trace of an isotopy or the attachment of an 
  $m$-handle for $m=0,1,2$. Denote by $K_{i-1}$ and $K_i$ the links that form 
  the bottom and top boundaries, respectively, of $F_i$. 

  Let $\gamma$ be an oriented smooth embedded path in $F$ that begins on $\leg 
  = K_0$, ends on $\leg' = K_n$, intersects each homologically non-trivial 
  component of the simple multicurve $K_i \subset F$ exactly once, and does not intersect 
  any null-homologous component. Such a path exists by an argument that 
  successively removes innermost arcs that arise when a candidate path 
  intersects a component of $K_i$ at least twice.  A consequence of this 
  construction is that if $F_i$ is the attachment of a $0$-handle (resp.\ a 
  $2$-handle) and $K_i^0$ is the component of $K_i$ created by the $0$-handle 
  (resp.\ $K_{i-1}^2$ is the component erased by the $2$-handle), then $\gamma$ 
  does not intersect $K_i^0$ (resp.\ $K_{i-1}^2$).  We say that a Legendrian 
  representative $\leg_i$ of $K_i$ is \dfn{$\gamma$-compatible} if the rotation 
  number is $r$ for all components that $\gamma$ intersects in an upward 
  direction with respect to the $[0,1]$-component, is $-r$ for all components 
  $\gamma$ intersects in a downward direction, and is $0$ otherwise.

  The proof now proceeds inductively, one elementary cobordism at a time, by 
  showing that if $\leg_{i-1}$ is a $\gamma$-compatible Legendrian 
  representative of the link $K_{i-1}$, then there exists a $\gamma$-compatible 
  Legendrian representative $\leg_i$ of $K_i$ with $\leg_{i-1} \qconcd \leg_i$.  
  We may start the inductive process since $K_0 = \leg$ is a 
  $\gamma$-compatible Legendrian knot by construction.  

  We prove the inductive claim in four cases:
  \begin{description}
    \item[Isotopy] Choose a Legendrian representative $\leg_i$ of $K_i$.  By 
      further stabilizing each component of $\leg_i$ appropriately, we may 
      assume that $\leg_i$ is $\gamma$-compatible.  In particular, we see that 
      corresponding components of $\leg_{i-1}$ and $\leg_i$ have the same 
      rotation numbers.  \fullref{cor:disconnected-smooth-isotopy} then implies 
      that $\leg_{i-1}$ and $\leg_i$ are disconnected Lagrangian zigzag 
      concordant.

    \item[$0$-handle] Simply take $\leg_i$ to be the result of attaching a 
      Lagrangian $0$-handle to $\leg_{i-1}$.  Since the component of $K_i$ 
      birthed by a $0$-handle cannot intersect $\gamma$, we conclude that 
      $\leg_i$ is $\gamma$-compatible if $\leg_{i-1}$ is.

    \item[$1$-handle]  Let $\beta: [-1,1] \times [-\delta,\delta] \to \R^3$ be 
      the embedded band along which the $1$-handle is attached, with 
      $\beta(\{-1,1\} \times [-\delta, \delta]) \subset \leg_{i-1}$.  Perform 
      an isotopy on $\beta$ supported in $(-2\epsilon, 2\epsilon) \times 
      [-\delta,\delta]$, for sufficiently small $\epsilon>0$, so that along 
      $[-\epsilon,\epsilon] \times [\delta,\delta]$, the result of the isotopy 
      matches the band in the standard Lagrangian $1$-handle (see \cite{Dim16:LegAmbSurg, EkhHonKal16:LagCob}. Call the 
      resulting embedding of the band $\beta'$. 

      Let $\leg'_{i-1}$ be a Legendrian link that agrees with $\leg_{i-1}$ off 
      of $\beta(\{-1,1\} \times [-\delta, \delta])$, contains 
      $\beta'(\{-\epsilon,\epsilon\} \times [-\delta,\delta])$, and $C^0$ 
      approximates the curves defined by $\beta'$ that join the endpoints of 
      the curves above. By construction, $\leg_{i-1}$ and $\leg'_{i-1}$ are 
      smoothly isotopic. Add small stabilizations to $\leg'_{i-1}$ if necessary 
      to ensure that rotation numbers of corresponding components of 
      $\leg_{i-1}$ and $\leg'_{i-1}$ agree.  By the first case, above, we know 
      that $\leg_{i-1}$ and $\leg'_{i-1}$ are disconnected Lagrangian zigzag 
      concordant.  

      We then let $\leg_i$ be the result of attaching the Lagrangian $1$-handle 
      defined by $\beta'$ to $\leg'_{i-1}$.  To verify that $\leg_i$ is the 
      desired Legendrian link, we first note that it is, indeed, a 
      representative of the smooth knot $K_i$ since it was obtained from 
      $K_{i-1}$ by attaching a $1$-handle along a band isotopic to the original 
      band.  Second, we check that $\leg_i$ is $\gamma$-compatible. The key 
      property here is that the sum of the rotation numbers of the components 
      on one end of the Lagrangian $1$-handle cobordism is equal to the 
      rotation number of the component on the other end.  We conclude that 
      $\leg_i$ is $\gamma$-compatible, as either $\gamma$ intersects two 
      components on one end of the cobordism in opposite directions (and hence 
      the rotation number on the other end, which cannot intersect $\gamma$, is 
      $r-r = 0$) or intersects one component on either end of the cobordism 
      (and hence the rotation number of the non-intersected component must be 
      $0$ while the rotation numbers of the other components are both $\pm r$). 

    \item[$2$-handle]  We first note that the component $\leg^2_{i-1}$ of 
      $\leg_{i-1}$ that is erased by the $2$-handle does not intersect 
      $\gamma$, and hence has rotation number $0$.  Further, the component 
      $\leg^2_{i-1}$ is unknotted and unlinked from the rest of $\leg_{i-1}$.  
      Thus, by \fullref{cor:disconnected-smooth-isotopy}, the link $\leg_{i-1}$ 
      is Lagrangian zigzag concordant to a Legendrian link $\leg'_{i-1}$ with 
      the component $\leg^2_{i-1}$ replaced by a maximal unknot.  Further, $\leg_i$ is just the link $\leg_{i-1}$ with $\leg^2_{i-1}$ removed.  Thus, by attaching a Lagrangian $0$-handle to $\leg_i$, we obtain a Lagrangian cobordism from $\leg_i$ to $\leg'_{i-1}$. It follows that we have 
      $\leg_{i-1} \qconcd \leg'_{i-1} \qconc \leg_i$. The 
      $\gamma$-compatibility of $\leg_i$ follows immediately.
  \end{description}

  This completes the inductive argument. Together with transitivity of the 
  relation $\qconcd$, this suffices to prove the proposition since $\leg_n$ is 
  a Legendrian representative of the knot $K_n = \leg'$ and the 
  $\gamma$-compatibility of $\leg_n$ implies that $\leg_n$ and $\leg'$ have the 
  same rotation number, so \fullref{cor:disconnected-smooth-isotopy} shows that 
  $\leg_n \qconcd \leg'$.
\end{proof}

\subsection{Comparison to Lagrangian cobordism}
\label{sssec:lqc-vs-lagr}

\fullref{rem:zz-genus-neq-smooth-genus}, \fullref{prop:zz-cob-classical}, and 
\fullref{thm:disconnected-lqc} show that the Lagrangian zigzag concordance 
relation is more rigid than smooth concordance.  On the other hand, 
\fullref{thm:lqc-not-c}, which we will prove in this section, shows that 
Lagrangian zigzag concordance is more flexible that Lagrangian concordance. In 
particular, we find a pair of Legendrian knots $\leg$ and $\leg'$ that are 
Lagrangian zigzag concordant but not Lagrangian concordant in either direction.  
Our pair of knots is the one in \fullref{ex:zz-not-c} and 
\fullref{fig:lqc-not-c}, which we already know to be Lagrangian zigzag 
concordant.

To show these knots are not Lagrangian concordant in either direction, the 
first ingredient is Pan's obstruction to Lagrangian cobordisms via ruling 
polynomials \cite{Pan17:LagCobAug}, which are a type of skein relation 
invariant of Legendrian knots.  See the survey article 
\cite{Sab21:RulingsIntro} for an introduction to the notion of a ruling of a 
Legendrian link.

\begin{theorem}[{\cite[Corollary~1.7]{Pan17:LagCobAug}}]
  \label{thm:ruling-ineq}
  Suppose that there exists a spin exact Maslov-$0$ Lagrangian concordance from 
  $\leg_-$ to $\leg_+$. The ruling polynomials satisfy the following inequality 
  for any prime power $q$:
  \[ R_{\leg_-} (q^{1/2} - q^{-1/2}) \leq R_{\leg_+} (q^{1/2} - q^{-1/2}). \]
\end{theorem}


Note that we may apply the theorem to obstruct concordances between any two 
Legendrians with rotation number $0$, as a Lagrangian concordance between two 
knots of rotation number $0$ automatically has Maslov number $0$; further, any 
orientable $2$-manifold has a spin structure that may be restricted to spin 
structures along the boundary.  


The second ingredient is the Legendrian satellite $\Sigma(\leg, \Pi)$ of a 
pattern $\Pi \subset J^1S^1$ around a companion $\leg \subset \R^3$ (see 
\cite{EtnVer18:LegSatellites, Ng01:LegSatellites, NgTra04:LegTorusLinks}) and 
Cornwell, Ng, and Sivek's construction of satellites of concordances 
\cite{CorNgSiv16:LagConcObstructions}.

\begin{theorem}[{\cite[Theorem~2.4]{CorNgSiv16:LagConcObstructions}}] 
  \label{thm:satellite-concordance}
  Let $\Pi$ be a Legendrian pattern in $J^1S^1$.  If $\leg_- \prec \leg_+$, 
  then $\Sigma(\leg_-, \Pi) \prec \Sigma(\leg_+, \Pi)$.
\end{theorem}

There is a generalization of this theorem to Lagrangian cobordisms in 
\cite{GuaSabYac22:LegSatellitesCob}, though we will not need the full power of 
the results therein.


\begin{proof}[Proof of \fullref{thm:lqc-not-c}]  

  As shown in \fullref{ex:zz-not-c}, the Legendrian knots $\leg$ and $\leg'$ 
  from \fullref{fig:lqc-not-c} are Lagrangian zigzag concordant.


  On the other hand, it suffices to show that the $\tb$-twisted Whitehead 
  doubles $\Wh(\leg)$ and $\Wh(\leg')$ are not Lagrangian concordant in either 
  direction using \fullref{thm:ruling-ineq}; the contrapositive of 
  \fullref{thm:satellite-concordance} then implies that $\leg$ and $\leg'$ are 
  also not concordant in either direction.  

  A tedious but straightforward computation shows that
  \begin{align*}
    R_{\Wh(\leg)}(z) &= 5+3z^2,\\
    R_{\Wh(\leg')}(z) &= 2+3z^2+3z^4+z^6.
  \end{align*}
  We then compute that
  \begin{align*}
    R_{\Wh(\leg)}\left(\sqrt{2} - 1/\sqrt{2}\right) &> 
    R_{\Wh(\leg')}\left(\sqrt{2} - 1/\sqrt{2}\right), \\
    R_{\Wh(\leg)}\left(\sqrt{3} - 1/\sqrt{3}\right) &< 
    R_{\Wh(\leg')}\left(\sqrt{3} - 1/\sqrt{3}\right).
  \end{align*}
  The result follows.
\end{proof}

%% file: sec_satellite.tex
\section{Global structure of the zigzag cobordism relation}
\label{sec:zz-graph}

In this section, we study the global structure, metric, and monoidal properties of the Lagrangian zigzag cobordism relation.  After constructing a weighted graph with edge metric that describes the zigzag cobordism relation, we explore
\begin{itemize}
\item The structure of the weighted graph, including its connectivity, 
  diameter, and links of vertices (\fullref{ssec:zz-graph});
\item The monoidal structure of the zigzag cobordism relation 
  (\fullref{ssec:zz-torsion}); and
\item Quasi-isometries of the graph given by the satellite construction 
  (\fullref{ssec:satellite}).
\end{itemize}

\subsection{The Lagrangian zigzag cobordism graph}
\label{ssec:zz-graph}

To frame the study of the global structure of the Lagrangian zigzag cobordism 
relation, we take inspiration from Cochran and Harvey's study 
\cite{CocHar18:ConcGeom} of the smooth knot concordance group as a metric group 
with distance defined by the minimal smooth relative $4$-genus $g_4$ between 
concordance classes of knots.  In particular, we make the following definition:

\begin{definition} \label{defn:zz-graph}
  The Lagrangian zigzag cobordism graph $\LZC$ is the weighted graph defined by 
  the following:
	\begin{itemize}
    \item Its vertices correspond to equivalence classes of oriented Legendrian 
      knots under the Lagrangian zigzag concordance relation $\qconc$.
		
		\item A single edge exists between equivalence classes $[\leg_1]$ and $[\leg_2]$ if $\leg_1 \qcob \leg_2$.
		
		\item The weight of an edge between $[\leg_1]$ and $[\leg_2]$ is $\LagGenus(\leg_1, \leg_2)$.
	\end{itemize}
%
\end{definition}

Since $\LagGenus$ is clearly subadditive and symmetric, the weighted graph 
$\LZC$ is also a metric space.

We will henceforth drop the brackets from the notation for equivalence classes of Lagrangian zigzag concordance, denoting such an equivalence class by a representative.


The remainder of this subsection is devoted to proving fundamental properties 
of the graph $\LZC$, as encapsulated in \fullref{prop:graph-structure}.  First, 
the connectivity of $\LZC$ is determined by the main result of 
\cite{SabVelWon21:MaxLeg}:

\begin{proposition} \label{prop:zz-connectivity}
  The graph $\LZC$ has exactly one connected component for each rotation 
  number.
\end{proposition}

Next, we show that each component of $\LZC$ has infinite diameter.  While this 
may be proven using the fact that the underlying smooth cobordism graph has 
infinite diameter with respect to $g_4$---for example, one may take connected 
sums with $(2,n)$-torus knots---the proposition below shows that $\LZC$ 
contains finer structure than the smooth concordance group as a metric space.

\begin{proposition}\label{prop:diameter}
  For any Legendrian $\leg$ and $n \in \N$, there exists a Legendrian $\leg_n$ 
  such that $\LagGenus (\leg,\leg_n) = n$ and $g_4(\leg, \leg_n) = 0$.
\end{proposition}

\begin{proof}
  Let $\leg'$ be the maximal Legendrian representative of the $m (6_1)$ knot, 
  pictured in \fullref{fig:6_1}.
  \begin{figure}[!htb]
    \includegraphics{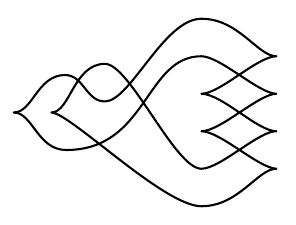}
    \caption{The maximal Legendrian $m (6_1)$ knot.}
		\label{fig:6_1}
	\end{figure}
  
  This knot is smoothly slice, but since its Thurston--Bennequin number is 
  $-3$, it is not Lagrangian slice.  Consider $\leg_n = \leg \connsum n \leg'$.  
  On one hand, this knot satisfies $g_4(\leg, \leg_n) =0$.  On the other hand, 
  we may compute that \[
    \abs{\tb(\leg) - \tb(\leg_n)} = 2n,
  \]
  and hence \fullref{prop:zz-cob-classical} implies that $\LagGenus(\leg, 
  \leg_n) \geq n$.
\end{proof}

Local properties of the $\LZC$ graph may be stated using the $1$-link of a 
vertex $\leg$, which we define to be the set of vertices that have distance 
exactly $1$ from $\leg$.

\begin{proposition}\label{prop:valence}
	Given a Legendrian knot $\leg$, the cardinality of its $1$-link is countably infinite.  In particular, there exists a sequence of Legendrian knots $\leg_i$ so that $\LagGenus(\leg,\leg_i) = 1$ but $\leg_i \not \qconc \leg_j$ for $i \neq j$.
\end{proposition}

\begin{proof}
  Given $i>0$, let $\leg'_i$ be the Legendrian representative of the pretzel 
  knot $K_i =P(-3, -5, -(2i+1))$. As illustrated in \fullref{fig:pretzel_knot}, 
  there is a genus~$1$ Lagrangian cobordism from the maximal Legendrian unknot 
  $\maxunknot$ to $\leg'_i$.  By \cite[Lemma 6.7]{SabVelWon21:MaxLeg}, we then 
  have $\LagGenus(\maxunknot,\leg'_i) = 1$.
  \begin{figure}[!htb]
    \labellist
    \small\hair 2pt
    \pinlabel {$2i+1$} [r] at 158 158
    \endlabellist
    \includegraphics{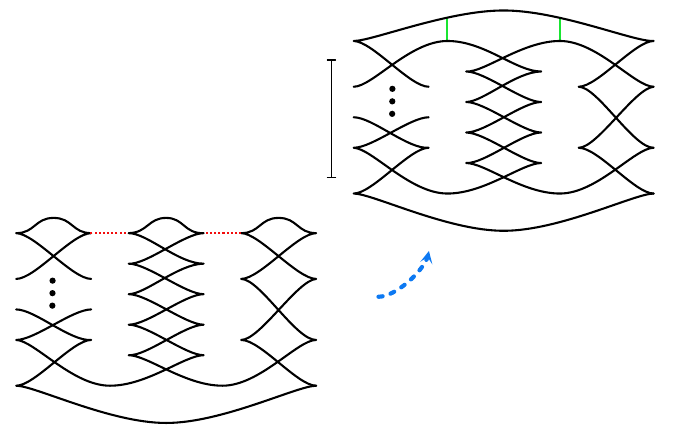}
		\caption{By attaching two $1$-handles to a maximal Legendrian unknot, we obtain a genus $1$ Lagrangian cobordism from that unknot to $\leg'_i$.}
		\label{fig:pretzel_knot}
	\end{figure}
	
	

  It is straightforward to compute that the Alexander polynomial of $K_i$ is 
  given by
  \[\Delta_{K_i}(t) =  (4i+6)t^2+(8i+11)t+4i+6,\]
  which has degree $2$ and is irreducible and distinct for each $i$. Then for 
  any $K_{s}$, $r\neq s$,
  \[\Delta_{K_r\connsum K_{s}}(t) = \Delta_{K_r}(t)\Delta_{K_{s}}(t) \neq 
    t^kf(t)f(t^{-1})
  \]
  for any polynomial $f$ and integer $k$. Thus no two such pretzel knots are smoothly concordant.

  Now, given a Legendrian $\leg$, let $\leg_i = \leg \connsum \leg'_i$. Then, 
  again using \cite[Lemma 6.7]{SabVelWon21:MaxLeg}, we have 
  $\LagGenus(\leg,\leg_i) = 1$ but $\leg_i \not \qconc \leg_j$ for $i \neq j$ 
  since they cannot be smoothly concordant.
\end{proof}


\begin{proof}[Proof of \fullref{prop:graph-structure}]
  The proposition is a combination of \fullref{prop:zz-connectivity}, 
  \fullref{prop:diameter}, and \fullref{prop:valence}.
\end{proof}

\subsection{Monoidal structure}
\label{ssec:zz-torsion}

In this section, we turn to the algebraic structure of $\LZC$.  The monoidal 
structure on $\LZC$ is defined by connected sum $\connsum$ 
\cite{EtnHon03:LegConnSum}.  That the connected sum operation is well defined 
on $\LZC$ follows from \cite[Corollary~5.1]{GuaSabYac22:LegSatellitesCob}.

In the smooth setting, the connected sum becomes a group operation when we 
descend to concordance classes: In particular, inverses exist since the 
connected sum of a knot and its mirror reverse is smoothly slice.  In the 
Legendrian setting, while we still have a unit given by the maximal unknot 
$\maxunknot$, and the connected sum operation is still associative, evidence 
points to the idea that inverses \emph{never} exist.  Broadly speaking, the 
obstruction is that Legendrian realizations of smooth mirrors are not 
well behaved. We discuss two indications that this is, indeed, the case.  

To set notation, let $\leg$ be a Legendrian knot and $\leg'$ be a Legendrian 
realization of its mirror reverse; of course, $\leg \connsum \leg'$ is smoothly 
slice.  First, if $\leg$ is stabilized, then so is $\leg \connsum \leg'$, and 
hence it cannot be Lagrangian concordant to $\maxunknot$. It is unclear, 
however, if $\leg \connsum \leg'$ is Lagrangian zigzag concordant to 
$\maxunknot$, especially as it is unknown whether stabilizations are preserved 
under Lagrangian zigzag concordance  (though we know that they are \emph{not} 
preserved under Lagrangian zigzag \emph{cobordism}).  Second, assuming that 
$\leg$ and $\leg'$ have maximal $\tb$, a necessary condition for $\leg \connsum 
\leg'$ to be Lagrangian slice would be for $\tb(\leg) = - \tb(\leg')$, which 
would imply that $\tb(\leg \connsum \leg') = -1$. This condition, however, is 
not satisfied for any knot of $10$ crossings or fewer \cite{LivMoo:KnotInfo}.  
That said, it may be possible to find a Legendrian realization $\leg''$ of a 
smooth knot concordant to $\leg'$ that has the correct $\tb$.

\begin{openquestion}
  \label{openq:inverses}
  Does any non-trivial Lagrangian zigzag concordance class have an inverse with 
  respect to the connected sum operation? If not, does the connected sum 
  operation have the cancellation property?
\end{openquestion}

As a first step in investigating the algebraic structure of the $\LZC$ monoid, 
we  study torsion (or the lack thereof). The underlying result relates the 
maximal Thurston--Bennequin number $\overline{\tb}$ and the arc index $\arcindex$ of a smooth knot 
and its mirror.

\begin{proposition}
	\label{prop:max-tb-grid}
  For any smooth knot $K$, we have \[\overline{\tb}(K) + \overline{\tb}(m(K)) = 
    - \arcindex(K).\]
\end{proposition}

Before beginning the proof, we need to set some notation.  If $\gd$ is a grid 
diagram for $K$, then we denote the ($\pi/2$)-rotation by $m(\gd)$, which is a 
grid diagram for $m(K)$.  Further, the Legendrian front associated to $\gd$ is 
denoted $\leg_\gd$.

\begin{proof}[Proof of \fullref{prop:max-tb-grid}]
  Let $\gd$ be a minimal grid diagram for $K$.  By 
  \cite[Corollary~3]{DynPra13:MinGridMaxTB}, the associated front diagram 
  $\leg_\gd$ realizes $\overline{\tb}(K)$.  The same is true for 
  $\leg_{m(\gd)}$ and $\overline{\tb}(m(K))$.
	
  When we compute $\tb(\leg_\gd) + \tb(\leg_{m(\gd)})$, the contributions from 
  the crossings cancel, as each positive (resp.\ negative) crossing in 
  $\leg_\gd$ is a negative (resp.\ positive) crossing in $\leg_{m(\gd)}$.  
  Thus, $\tb(\leg_\gd) + \tb(\leg_{m(\gd)})$ is equal to half the total number 
  of cusps in $\leg_\gd$ and $\leg_{m(\gd)}$.  Each $X$ or $O$ in the grid 
  diagram $\gd$ yields a cusp in exactly one of $\leg_\gd$ or $\leg_{m(\gd)}$, 
  and hence we see that the total number of cusps in $\leg_\gd$ and 
  $\leg_{m(\gd)}$ is equal to twice $\arcindex(K)$.
\end{proof}

\begin{corollary}
	\label{cor:amph-not-torsion}
	No non-trivial amphicheiral Legendrian knot is $2$-torsion in $\LZC$.
\end{corollary}

\begin{proof}
  Suppose that $\leg$ is amphicheiral and $2$-torsion.  Since $\leg$ is 
  amphicheiral, \fullref{prop:max-tb-grid} shows that $\arcindex(K) \leq 
  -2\tb(\leg)$.  On the other hand, since $\leg$ is $2$-torsion, we have $-1 = 
  \tb(\leg \connsum \leg) = 2\tb(\leg) +1$, and hence that $2 \tb(\leg) = -2$.  
  Thus, we see that $\arcindex(\leg) \leq 2$, and hence that $\leg$ is an unknot.
\end{proof}

%

\subsection{Metric structure}
\label{ssec:satellite}

In this subsection, we state some properties of $\LZC$ as a metric space (or 
equivalently, a weighted graph). The main theorems here are inspired by their 
smooth analogues in \cite{CocHar18:ConcGeom}, and we point the reader to 
\cite{CocHar18:ConcGeom} for relevant definitions, e.g.\ of quasi-isometry and 
related terms.  We begin by establishing the non-hyperbolicity of $\LZC$.

\begin{theorem}[{cf.\ \cite[Theorem~4.1]{CocHar18:ConcGeom}}]
  Given $n \geq 1$, there exists a subspace of $\LZC$ that is quasi-isometric 
  to $\R^n$. Consequently, $\LZC$ cannot be isometrically embedded in a finite 
  product of $\delta$-hyperbolic spaces.
\end{theorem}

\begin{proof}
  In the proof of \cite[Theorem~4.1]{CocHar18:ConcGeom}, knots $K_1, \dotsc, 
  K_n$ that are linearly independent in the smooth concordance group 
  $\mathcal{C}$ are chosen, together with Tristram signature functions 
  $\sigma_j \colon \mathcal{C} \to \Z$, $1 \leq j \leq n$, that satisfy
  \begin{equation}
    \label{eqn:tristram}
    g_4 (K) \geq \frac{1}{2} \abs{\sigma_j (K)} \text{ for all } K, \qquad
    \sigma_j (K_i) = 2 \delta_{ij}.
  \end{equation}
  The goal there is then to prove, for $\vec{x} = (x_1, \dotsc, x_n), \vec{y} = 
  (y_1, \dotsc, y_n) \in \Z^n$, the inequalities
  \begin{equation}
    \label{eqn:non-hyp-ineq}
    \frac{1}{n} \taxicab (\vec{x}, \vec{y}) \leq g_4 (\vec{x}, \vec{y}) \leq n \taxicab 
    (\vec{x}, \vec{y}),
  \end{equation}
  where $\taxicab$ is the taxicab metric on $\Z^n$, and $g_4 (\vec{x}, \vec{y})$ is 
  the slice genus distance between $x_1 K_1 \connsum \dotsb \connsum x_n K_n$ 
  and $y_1 K_n \connsum \dotsb \connsum y_n K_n$. To complete that proof, the 
  left inequality of \eqref{eqn:non-hyp-ineq} follows from 
  \eqref{eqn:tristram}, while the right inequality follows from the 
  subadditivity and symmetry of metrics.

  In the present context, we follow the same proof, limiting $\vec{x}$ and 
  $\vec{y}$ to $\Z_{\geq0}^n$ and choosing Legendrian representatives $\leg_i$ 
  of $K_i$.  Then the Legendrians $\leg_i$ are guaranteed to be linearly 
  independent in $\LZC$, and the goal is to prove that
  \begin{equation}
    \label{eqn:non-hyp-ineq-lag}
    \frac{1}{n} \taxicab (\vec{x}, \vec{y}) \leq \LagGenus (\vec{x}, \vec{y}) 
    \leq n \taxicab (\vec{x}, \vec{y}),
  \end{equation}
  where $\LagGenus (\vec{x}, \vec{y})$ is the relative Lagrangian genus between 
  $x_1 \leg_1 \connsum \dotsb \connsum x_n \leg_n$ and $y_1 \leg_1 \connsum 
  \dotsb \connsum y_n \leg_n$.  The left inequality of 
  \eqref{eqn:non-hyp-ineq-lag} follows from \eqref{eqn:tristram} as in the 
  proof of \cite[Theorem~4.1]{CocHar18:ConcGeom}, with the additional 
  observation that $\LagGenus (\leg) \geq g_4 (\Lambda)$.
\end{proof}

Next, we prove \fullref{thm:winding-1-qis}, which states that 
winding-number-$\pm 1$ satellite operators with rotation number $0$ relative to 
the identity or reverse operator are self-quasi-isometries of $\LZC$.
To clarify this statement, every satellite operator corresponds to a pattern 
Legendrian link in $J^1 S^1$; for example, the identity operator corresponds to 
the core of $J^1 S^1$.  A pair of homologous links $\Pi_1, \Pi_2 \in J^1 S^1$ 
cobound a Seifert surface $S$, meaning that $\bdy S = \Pi_1 \disjunion -\Pi_2$, 
giving a relative rotation number $\rot (\Pi_1 \disjunion -\Pi_2)$ when they 
are both Legendrian.  \emph{A priori,} this relative rotation number depends on 
the homology class of the Seifert surface; but since $H_2 (J^1 S^1) = 0$, it is 
in fact well defined.
%

As in \cite{CocHar18:ConcGeom}, we use the following lemma to transform the 
problem into proving that such operators are a bounded distance from known 
quasi-isometries.



\begin{lemma}[{\cite[Lemma~6.2]{CocHar18:ConcGeom}}]
  \label{lem:quasi-isom}
  Let $(X, d)$ and $(X', d')$ be two metric spaces. Suppose that $f \colon X 
  \to X'$ is a quasi-isometric embedding.  If $g \colon X \to X'$ is within a 
  bounded distance from $f$, then $g$ is a quasi-isometric embedding.  Further, 
  if $f$ is a quasi-isometry, then so is $g$.
\end{lemma}


To prove the bounded distance,
we first extend \cite[Proposition~3.1]{SabVelWon21:MaxLeg} to links in $J^1 
S^1$ that are not necessarily null-homologous.

\begin{lemma}[{\cite[Proposition~3.1]{SabVelWon21:MaxLeg}}]
  \label{lem:lqc-satellite-general}
  Suppose that $\Lambda_1$ and $\Lambda_2$ are Legendrian links in $(M,\alpha)$ 
  with $[\Lambda_1] = [\Lambda_2] \in H_1(M)$ and relative rotation number 
  $r_{[\Sigma]}(\Lambda_1 \disjunion -\Lambda_2) = 0$ with respect to some 
  Seifert surface $\Sigma$ for $\Lambda_1 \disjunion -\Lambda_2$. Then there 
  exists a Lagrangian zigzag cobordism between $\Lambda_1$ and $\Lambda_2$.
\end{lemma}

\begin{proof}
  The proof of \fullref{lem:lqc-satellite-general} follows an argument that 
  closely mirrors the proof of \cite[Lemma~3.5]{SabVelWon21:MaxLeg}.  In what 
  follows, we discuss the basic strategy but point the reader to 
  \cite{SabVelWon21:MaxLeg} for a more detailed
	treatment.
	
	We can assume without loss of generality that both $\Lambda_1$ and $\Lambda_2$ 
	are knots. If that were not the case, then one can perform Legendrian surgery along a 
	collection of Legendrian arcs in their complement to separately join each of their 
	components.
	
  We now consider a Seifert surface $\Sigma$ for the link $\Lambda_1 \disjunion 
  -\Lambda_2$. After possibly double-stabilizing $\Lambda_1$ and $\Lambda_2$ to 
  achieve negative twisting of $\xi$ along each component of $\partial \Sigma$, 
  we can isotope $\Sigma$ relative to the boundary to be convex.
	
	The idea now is to view $\Sigma$ in disk--band form by choosing an arc-basis 
	$\set{a_1, \dotsc, a_g, a_{g+1}}$ for $\Sigma$ consisting of a collection of properly 
  embedded arcs in $\Sigma$, such that $\Lambda_2$ intersects only $a_{g+1}$ 
  and does so in a single point. The proof of 
  \cite[Lemma~3.5]{SabVelWon21:MaxLeg} details how, via Legendrian ambient 
  surgery (and therefore Lagrangian cobordism), one can decompose $\Sigma$ by 
  cutting along the arcs $a_1, \dotsc, a_g$. 
	
	At the end of the decomposition process, we obtain a Legendrian knot $\Lambda$ which 
	is Lagrangian cobordant to $\Lambda_1$. Moreover, $\Lambda$ and $\Lambda_2$ cobound a 
	convex annulus $A$ (the portion of $\Sigma$ bounded by $\Lambda$ and $\Lambda_2$). 
	This implies that $\Lambda$ and $\Lambda_2$ are smoothly isotopic. By 
	construction, $r_{[A]}(\Lambda \disjunion -\Lambda_2) = 0$, implying that after 
	sufficiently many double stabilizations of each, $\Lambda$ and $\Lambda_2$ become 
	Legendrian isotopic. In turn, $\Lambda_2$ is Lagrangian zigzag cobordant to $\Lambda$,
	which is itself Lagrangian zigzag cobordant to $\Lambda_1$, finishing the argument.
\end{proof}

We now state an immediate corollary that will be useful in the arguments that follow. Let
$C_{n,1}$ be the Legendrian in $J^1 S^1$ depicted in \fullref{fig:cn1}. Denote by 
$C^m_{n,1}$ the Legendrian knot obtained by stabilizing $C_{n,1}$ either 
positively or negatively $|m|$ times, depending on the sign of $m$.
%
\begin{figure}[!htb]
 \includegraphics[scale=1.0]{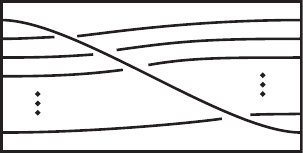}
  \caption{The Legendrian knot $C_{n,1} \in J^1 S^1$.}
  \label{fig:cn1}
\end{figure}

\begin{corollary}
  \label{cor:lqc-satellite-j1s1}
  A Legendrian link $\Pi$ in $J^1 S^1$ is Lagrangian zigzag cobordant to 
  $C^m_{n,1}$ if and only if $n$ is the winding number of $\Pi$ and $m$ is the 
  unique number for which $\rot (\Pi \disjunion -C_{n,1}^m) = 0$.
\end{corollary}

\begin{proof}
  Obviously,
  $[\Pi] = [C_{n,1}] \in H_1(J^1 S^1)$. It is also clear that the number $m$ as 
  described as unique. By \fullref{lem:lqc-satellite-general}, the condition 
  $\rot (\Pi \disjunion -C_{n,1}^m)$ implies that $\Pi$ and $C^m_{n,1}$ are 
  Lagrangian zigzag cobordant.
\end{proof}

\begin{proposition}[cf.~{\cite[Proposition~6.3]{CocHar18:ConcGeom}}]
  \label{prop:satellite-bounded-dist}
  Suppose that $\Pi \colon (\LZC, \LagGenus) \to (\LZC, \LagGenus)$ is a 
  satellite operator of winding number $n$. Then $\Pi$ is a bounded distance 
  from the cabling operator $C_{n,1}^m$, where $m$ is the unique number for 
  which $\rot (\Pi \disjunion -C^m_{n,1}) = 0$.
\end{proposition}

\begin{proof}
  The proof of \cite[Proposition~6.3]{CocHar18:ConcGeom} considers the two
  pattern links associated to the two satellite operators, which are 
  homologous in the solid torus.  There, the key observation is that the 
  minimum genus among all compact oriented surfaces cobounded by the two 
  pattern links depends only on $P$. The proof then concludes by embedding 
  the solid torus into $S^3$ as a tubular neighborhood of any given companion 
  $J$.

  We adapt this proof to the present context as follows. First, $\Pi$ 
  corresponds to a pattern Legendrian link $\Pi \subset J^1 S^1$, with $[\Pi] 
  = n \in H_1 (J^1 S^1)$, and the $C^m_{n,1}$-cabling operator 
  corresponds to the pattern Legendrian link $C^m_{n,1} \subset J^1 S^1$.
  By \fullref{cor:lqc-satellite-j1s1}, $\Pi$ is Lagrangian zigzag cobordant 
  to $C_{n,1}^m$.
  One difference from the paragraph above is that, instead of one surface 
  cobounded by $\Pi$ and $C_{n,1}^m$, we now have an interpolating zigzag of 
  Lagrangian surfaces.  Specifically, there exist Legendrian links
  \[
    \Pi = \leg_0, \leg_1, \dotsc, \leg_{k-1}, \leg_k = C_{n,1}^m
  \]
  and also $\leg_i^+$ for $i = 1, \dotsc, k$, together with Lagrangian 
  cobordisms
  \(
    L^<_i \subset \R \cross J^1 S^1
  \)
  from $\leg_{i-1}$ to $\leg_{+,i}$, and cobordisms
  \(
    L^>_i \subset \R \cross J^1 S^1
  \)
  from $\leg_i$ to $\leg_{+,i}$.
  We observe here that the minimum total genus $D$ among all choices of 
  Lagrangian zigzag cobordisms also depends only on $\Pi$. The proof 
  concludes now by embedding each $L^<_i, L^>_i \subset \R \cross J^1 S^1$ 
  into $\R \cross \R^3$
  using \fullref{thm:satellite-concordance},
  to obtain a Lagrangian zigzag cobordism of genus $D$ between $\Pi (J)$ and 
  $C_{n,1}^m (J)$.
\end{proof}

\begin{corollary}
  \label{cor:satellite-bounded-dist}
  Any winding-number-$1$ satellite operator with rotation number $0$ relative 
  to the identity operator is a bounded distance from the identity operator.  
  Any winding-number-$(-1)$ satellite operator with rotation number $0$ 
  relative to the reverse operator is a bounded distance from the reverse 
  operator.  
\end{corollary}

\begin{proof}
  This follows from \fullref{prop:satellite-bounded-dist} and the definition 
  of relative rotation numbers.
\end{proof}

\begin{proof}[Proof of \fullref{thm:winding-1-qis}]
  This is a direct consequence of \fullref{lem:quasi-isom} and 
  \fullref{cor:satellite-bounded-dist}.
\end{proof}

%

%% file: sec_maslov.tex

\section{\texorpdfstring{Maslov-$0$}{Maslov-zero} Lagrangian zigzag cobordism}
\label{sec:maslov-0}

The key result in this section is \fullref{thm:maslov-0}, a Maslov-$0$ 
refinement of the main theorem of \cite{SabVelWon21:MaxLeg} when the ambient 
contact manifold is the standard contact $\R^3$.  The result is not only 
interesting in its own right, but also facilitates applications of Legendrian 
contact homology to questions of Lagrangian zigzag cobordisms, as we shall see 
in the next section.

The proof parallels that in \cite[Section~4]{SabVelWon21:MaxLeg}:  we first 
find a Legendrian $\leg_-$ and Lagrangian cobordisms $L_-$ from $\leg_-$ to 
$\leg$ and $L_-'$ from $\leg_-$ to $\leg'$.  We record these cobordisms using 
``handle graphs'' $\graph$ and $\graph'$ on $\leg_-$ (see 
\cite[Section~2.4]{SabVelWon21:MaxLeg} and below), and then attach the handles 
in both $\graph$ and $\graph'$ to $\leg_-$ to create the desired $\leg_+$.  The 
 new step in this version of the proof is to take additional care to ensure that the handle graphs encode 
Maslov-$0$ Lagrangian cobordisms.

\subsection{\texorpdfstring{Maslov-$0$}{Maslov-zero} handle graphs}
\label{ssec:maslov-0-graph}

A Legendrian handle graph in the standard contact $\R^3$ 
is defined to be a pair $(\graph, \leg)$, where $\graph$ is a trivalent 
Legendrian graph and $\leg \subset \graph$ is a Legendrian link so that the 
vertices of $\graph$ all lie on $\leg$ and the set $\handles$ of edges not in 
$\leg$ is a finite collection of pairwise disjoint Legendrian arcs.  As 
described in \cite{SabVelWon21:MaxLeg}, work of Dimitroglou Rizell 
\cite{Dim16:LegAmbSurg} or, in this case, Ekholm, Honda, and K\'alm\'an 
\cite{EkhHonKal16:LagCob}, implies that performing Legendrian ambient surgery 
on a subset $\handles_0 \subset \handles$ yields an exact Lagrangian cobordism $L(\graph, \leg, 
\handles_0)$ from $\leg$ to the Legendrian $\surg(\graph, \leg, \handles_0)$.  If all edges $\handles$ 
are used in the ambient surgery, we simply use the notation 
$\surg(\graph,\leg)$.  The goal of this section is to refine these ideas to keep 
track of Maslov indices.

We build on the notion of Maslov potentials for Legendrian graphs developed 
by \cite{AnBae20:LegGraphsDGA} and exposited for front diagrams in 
\cite{AnBaeKal22:LegGraphsDGA}.  Denote the vertices of $\graph$ by $\vertices$ 
and the cusps of the front diagram of $\graph$ by $\cusps$.  An  
\dfn{$m$-graded Maslov potential} on the front diagram of a Legendrian handle 
graph $(\graph, \leg)$ is a function $\mu$ from the components of $\graph 
\setminus (\vertices \cup \cusps)$ to $\Z/m$ that satisfies \[\mu(u) = 
  \mu(l)+1\]
when the strands $u$ and $l$ meet at a cusp with the $z$ values of $l$ lower 
than those of $u$.  See \fullref{fig:maslov-vertices}~(a).  It is 
straightforward to check that the front diagram of a Legendrian link whose 
components have vanishing rotation numbers has a $0$-graded Maslov potential.

\begin{definition} \label{defn:maslov-0-graph}
  A \dfn{Maslov-$m$ handle graph} is a handle graph $(\graph, \leg)$ together 
  with an $m$-graded Maslov potential $\mu$ on its front diagram such that, at 
  each vertex of $\graph$, the values of $\mu$ of the adjacent edges are as in 
  \fullref{fig:maslov-vertices}~(b). We denote a Maslov-$m$ handle graph by a 
  triple $(\graph, \leg, \mu)$.

  \begin{figure}[!htb]
    \labellist
    \small\hair 2pt
    \pinlabel {$a+1$} [bl] at 27 48
    \pinlabel {$a$} [tl] at 27 29
    \pinlabel {(a)} [ ] at 27 5
    \pinlabel {$a+1$} [bl] at 133 48
    \pinlabel {$a$} [tl] at 133 29
    \pinlabel {$a$} [] at 166 44
    \pinlabel {$a$} [bl] at 206 49
    \pinlabel {$a$} [tl] at 206 31
    \pinlabel {$a$} [bl] at 276 48
    \pinlabel {$a-1$} [tl] at 276 30
    \pinlabel {(b)} [ ] at 212 5
    \endlabellist
    \includegraphics{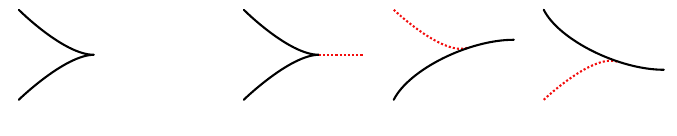}
    \caption{(a) A Maslov potential near a cusp. (b) A Maslov potential of the 
      front diagram of a Maslov-$m$ handle graph must satisfy three 
      compatibility conditions at the vertices.  The link $\leg$ is indicated 
      by solid arcs, while the handles are indicated by dotted arcs.}
    \label{fig:maslov-vertices}
  \end{figure}
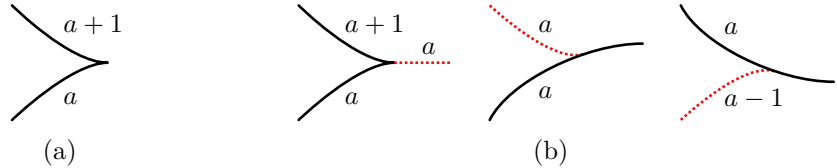
\end{definition}

\begin{lemma} \label{lem:maslov-0-well-def}
  The property of being a Maslov-$m$ handle graph is invariant under Legendrian 
  isotopy.
\end{lemma}

\begin{proof}
  The proof is a straightforward, if somewhat tedious, case-by-case check of 
  invariance under the six Reidemeister moves for Legendrian graphs in 
  \cite{ODoPav12:LegGraphs}.
\end{proof}

A vertex $v$ of the front diagram of a handle graph $(\graph, \leg)$ may be 
classified as either \dfn{smooth} if $v$ is a smooth point of the diagram of $\leg$ or 
\dfn{cusped} if $v$ is a cusp of the diagram. We say that a front diagram of a 
handle graph has \dfn{flat handles} if all vertices are cusped and there are no 
cusps on the interiors of the handles; see \fullref{fig:flat-handle} for an 
example.
\begin{figure}[!htb]
	\includegraphics{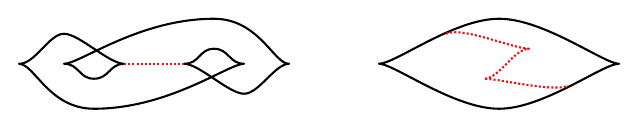}
	\caption{Legendrian handle graphs with (a) a flat handle and (b) a non-flat handle, both because the handle has cusps on the interior and because the vertices at its ends are smooth.}
	\label{fig:flat-handle}
\end{figure}

\subsection{Surgery on handle graphs}
\label{ssec:surgery-handle-graph}

We now have sufficient material to discuss surgery on handle graphs.

\begin{lemma}
  \label{lem:maslov-0-flat-handle}
  Any Maslov-$m$ handle graph $(\graph, \leg, \mu)$ is Legendrian isotopic to a 
  Maslov-$m$ handle graph whose front diagram has flat handles.
\end{lemma}

\begin{proof}
  Convert every smooth vertex of $\graph$ to a cusp vertex as in 
  \fullref{fig:flat-handle-pf}~(a), then eliminate all cusps on handles as in 
  \fullref{fig:flat-handle-pf}~(b). The result is a handle graph whose front 
  diagram has flat handles.
  \begin{figure}[!htb]
    \labellist
    \small\hair 2pt
    \pinlabel {$a$} [bl] at 22 152
    \pinlabel {$a$} [b] at 78 145
    \pinlabel {$a$} [b] at 78 54
    \pinlabel {$a+1$} [b] at 49 54
    \pinlabel {$a$} [tl] at 20 28
    \pinlabel {$a$} [bl] at 20 63
    \pinlabel {$a+1$} [bl] at 188 157
    \pinlabel {$a$} [tl] at 206 130
    \pinlabel {$a+1$} [bl] at 152 135
    \pinlabel {$a$} [tl] at 152 113
    \pinlabel {$a+1$} [bl] at 250 108
    \pinlabel {$a+1$} [br] at 247 87
    \pinlabel {$a$} [tl] at 266 73
    \pinlabel {$a+1$} [bl] at 305 61
    \pinlabel {$a+2$} [b] at 354 55
    \pinlabel {$a+1$} [br] at 323 33
    \pinlabel {$a$} [tl] at 336 18
    \pinlabel {RVI} [l] at 50 93
    \pinlabel {RIV} [tr] at 188 90
    \pinlabel {RVI} [tr] at 258 38
    \endlabellist
    \includegraphics{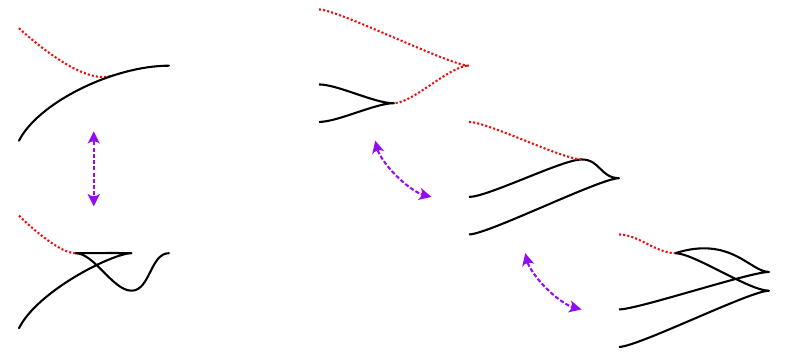}
    \caption{Converting a smooth vertex to a cusp vertex, left, and eliminating a cusp on a handle, right.}
    \label{fig:flat-handle-pf}
  \end{figure}
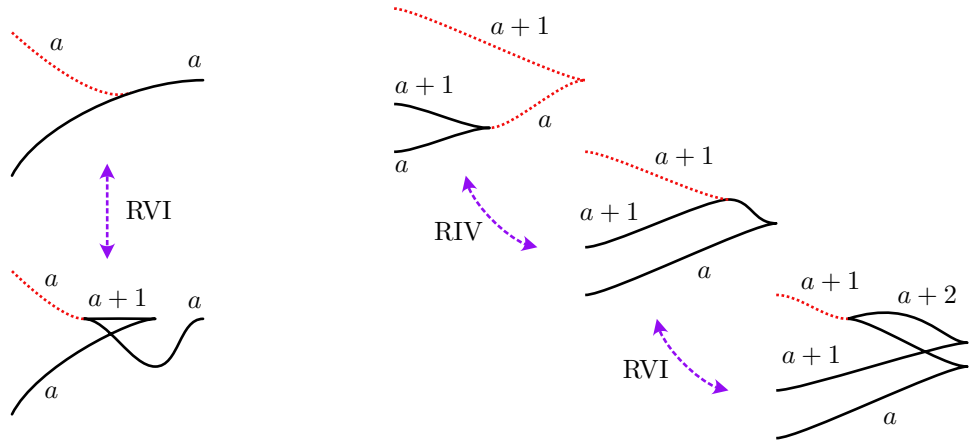
\end{proof}

\begin{lemma} \label{lem:flat-handle-lagr}
	Legendrian ambient surgery on a Maslov-$0$ handle graph $(\graph, \leg)$ yields a Maslov-$0$ Lagrangian cobordism from $\leg$ to $\surg(\graph, \leg)$.
\end{lemma}

\begin{proof}
  \fullref{lem:maslov-0-flat-handle} allows us to assume that after an isotopy, 
  all handles are flat.  As shown in \fullref{fig:maslov-0-surface}, both 
  Legendrian isotopy and attachment of a flat $1$-handle generate fronts for 
  Legendrian lifts of cobordisms with $0$-graded Maslov potentials.  It follows 
  that any loop on the front diagram for $L$ passes through the same number of 
  cusps in an upward direction as it does in a downward direction.  As in 
  \cite[Section~2.2]{RutSul20:CellularLCH1}, this shows that every loop has 
  Maslov index $0$, and hence that the cobordism $L$ has Maslov number $0$.
  \begin{figure}[!htb]
    \labellist
    \small\hair 2pt
    \pinlabel {$a+1$} [l] at 74 73
    \pinlabel {$a$} [l] at 75 53
    \pinlabel {$a+1$} [l] at 76 28
    \pinlabel {$a$} [l] at 76 11
    \pinlabel {$a+1$} [br] at 160 51
    \pinlabel {$a$} [tl] at 217 28
    \endlabellist
    \includegraphics{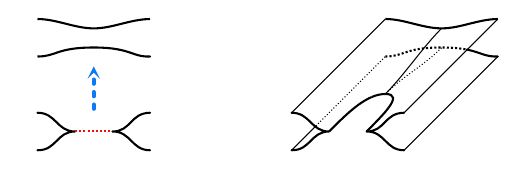}
    \caption{Surgery on a flat handle in a Maslov-$0$ handle graph yields a $0$-graded Maslov potential on the front of the associated Lagrangian cobordism.}
    \label{fig:maslov-0-surface}
  \end{figure}
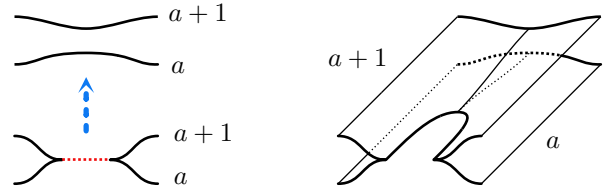
\end{proof}

\begin{remark}
  A version of this result is well known to experts, namely that a decomposable 
  Lagrangian cobordism has Maslov number $0$ whenever
    \begin{enumerate*}[label=(\alph*)]
      \item its $1$-handles are attached along flat handles; and
      \item these flat handles are between cusps whose incident strands have 
        matching Maslov potentials.
    \end{enumerate*}
  The purpose of the handle graph language introduced above is mainly to enable 
  work with handle graphs without flat handles in the next subsection.
\end{remark}



\subsection{\texorpdfstring{Maslov-$0$}{Maslov-zero} handle graphs for zigzag 
  cobordisms}
\label{ssec:maslov-0-qc}

The next steps in the proof of \fullref{thm:maslov-0} parallel those in 
\cite[Section~4]{SabVelWon21:MaxLeg}, making the diagrammatic operations in 
\cite[Lemmas~3.2, 4.2, 4.3, 4.4]{SabVelWon21:MaxLeg} that take a Legendrian 
$\leg$ and produce a handle graph $\graph$ on a stabilized unknot $\leg_-$ with 
$\leg = \surg(\graph, \leg_-)$ into Maslov-$0$ operations.  This procedure 
involves the operation of pinching between two parallel strands of a front 
diagram to create a handle in a handle graph. We note that pinching two 
adjacent strands in a Maslov-$m$ handle graph produces another Maslov-$m$ 
handle graph if the Maslov potential of the higher strand is (modulo $m$) one 
more than that of the lower strand.

The necessary modifications to the lemmas in \cite{SabVelWon21:MaxLeg} are 
described in the following two lemmas.  The first of these constructs a 
Maslov-$0$ cobordism from a double-stabilization to the original link; this 
lemma replaces \cite[Lemma~3.2]{SabVelWon21:MaxLeg}.

\begin{lemma} \label{lem:maslov-0-stab}
  Given a Legendrian link $\leg$ all of whose components have vanishing 
  rotation number, there exists a Maslov-$0$ handle graph $\graph$ on 
  $S_{+-}(\leg)$ with $\leg = \surg(\graph, S_{+-}(\leg))$.
\end{lemma}

\begin{proof}
The proof is contained in \fullref{fig:maslov-0-stab}, with the condition that 
the rotation number is $0$ yielding a $0$-graded Maslov potential on $\leg$, 
and hence on $S_{+-}(\leg)$.
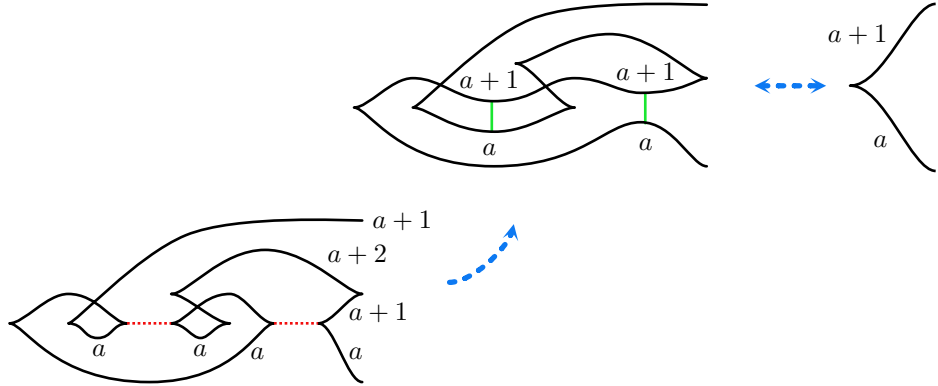
\begin{figure}[!htb]
  \labellist
  \small\hair 2pt
  \pinlabel {$a$} [l] at 94 15
  \pinlabel {$a$} [t] at 39 20
  \pinlabel {$a$} [t] at 77 20
  \pinlabel {$a+1$} [l] at 140 65
  \pinlabel {$a$} [l] at 131 15
  \pinlabel {$a+1$} [l] at 131 30
  \pinlabel {$a+2$} [l] at 123 52
  \pinlabel {$a+1$} [r] at 338 135
  \pinlabel {$a$} [r] at 338 95
  \pinlabel {$a+1$} [b] at 245 115
  \pinlabel {$a$} [t] at 245 98
  \pinlabel {$a+1$} [b] at 186 112
  \pinlabel {$a$} [t] at 186 96
  \endlabellist
  \includegraphics{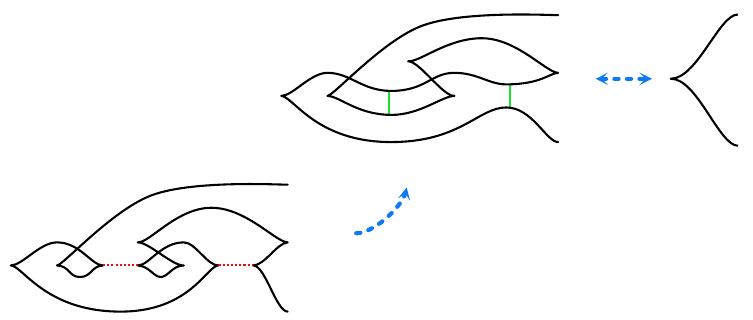}
  \caption{A handle graph on $S_{+-}(\leg)$ that yields $\leg$ under Maslov-$0$ surgery.}
  \label{fig:maslov-0-stab}
\end{figure}
\end{proof}

The next lemma allows us to replace an orientable $1$-handle attachment from $\leg_-$ to $\leg_+$ with a Maslov-$0$ $1$-handle attachment from a stabilization of $\leg_-$ to $\leg_+$.  We say that an ambient surgery along a flat handle is \dfn{ordered} if the Maslov potential at the top of the co-core of the handle in $\leg_+$ is larger than the Maslov potential at the bottom.

\begin{lemma} \label{lem:maslov-0-surgery}
  Suppose the Legendrian links $\leg_-$ and $\leg_+$ have components with 
  vanishing rotation number.  If there is a Lagrangian cobordism from $\leg_-$ 
  to $\leg_+$ induced by surgery on a single ordered $1$-handle in a handle 
  graph on $\leg_-$, then for some $k \geq 1$, there exists a Maslov-$0$ 
  Lagrangian cobordism from $(S_{+-})^k (\leg_-)$ to $\leg_+$ induced by 
  surgery on a Maslov-$0$ handle graph on $(S_{+-})^k (\leg_-)$.  Further, the front for $(S_{+-})^k (\leg_-)$ has the same crossings as that of $\leg_-$.
\end{lemma}

\begin{proof}
The proof is contained in \fullref{fig:maslov-0-surgery}.  Using a technique of 
\cite{FucTab97:LegTransKnots}, we may assume that, after a Legendrian isotopy, 
the front diagrams of $(S_{+-})^k (\leg_-)$ and $\leg_-$ are identical outside 
a small neighborhood located in a place of our choosing in which the 
stabilizations are performed.  In particular, after that isotopy, the front 
diagrams of $(S_{+-})^k (\leg_-)$ and $\leg_-$ will have the same number of 
crossings.
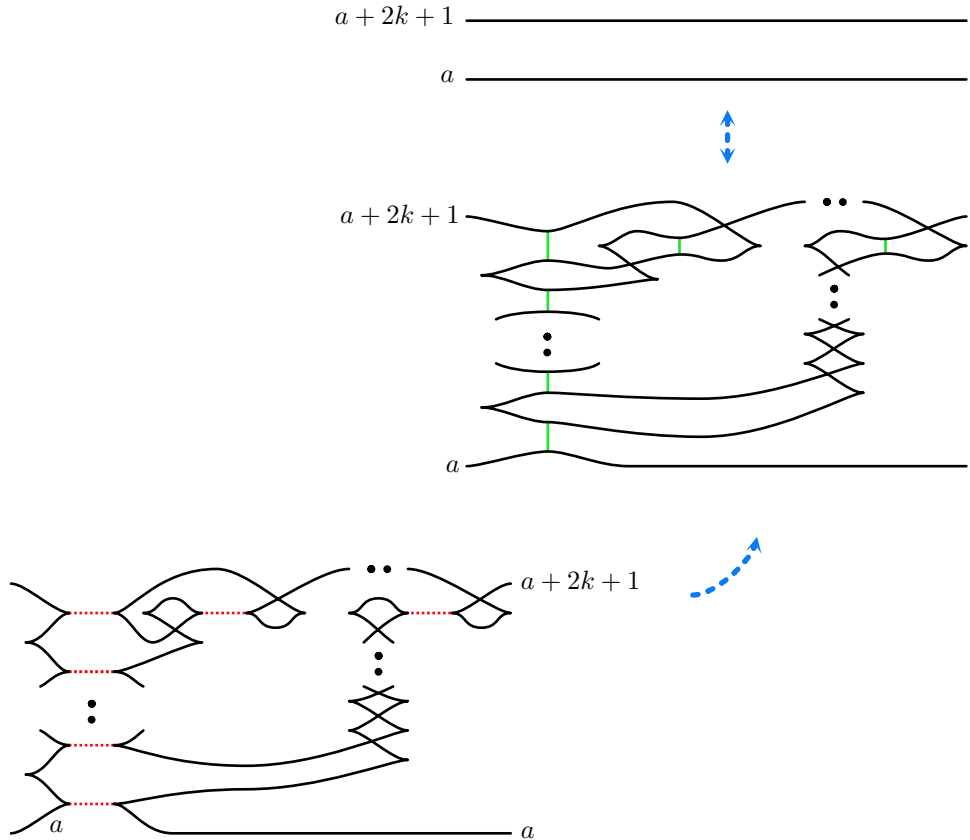
\begin{figure}[!htb]
\labellist
\small\hair 2pt
 \pinlabel {$a$} [tl] at 20 15
 \pinlabel {$a$} [l] at 197 9
 \pinlabel {$a+2k+1$} [l] at 197 104
 \pinlabel {$a$} [r] at 178 147
 \pinlabel {$a+2k+1$} [r] at 178 242
 \pinlabel {$a$} [r] at 176 294
 \pinlabel {$a+2k+1$} [r] at 176 317
\endlabellist
  \includegraphics{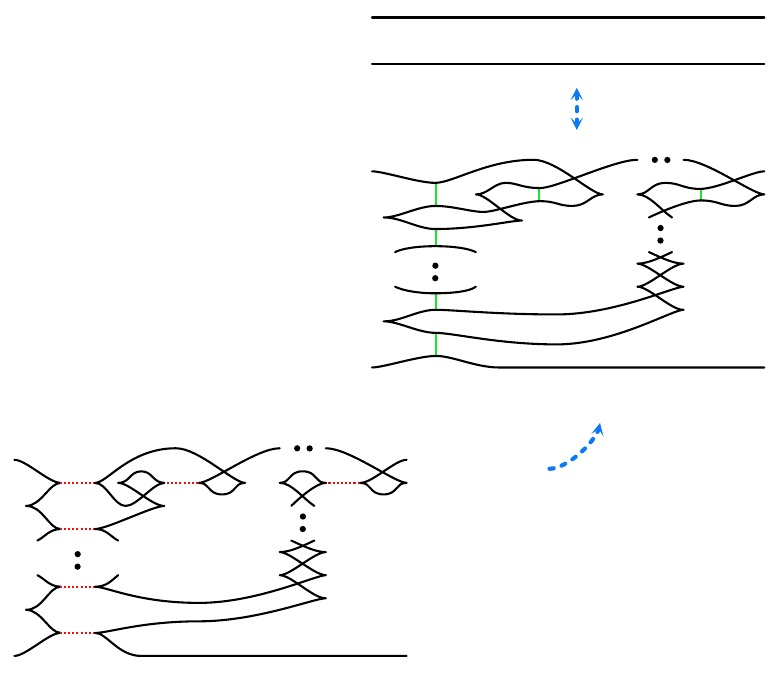}
	\caption{A handle graph on $S_{+-}^k(\leg_-)$ that yields $\leg_+$ under a Maslov-$0$ surgery, where $\leg_+$ is obtained from $\leg_-$ by a single orientable $1$-handle attachment.}
	\label{fig:maslov-0-surgery}
\end{figure}
\end{proof}

\begin{proof}[Proof of \fullref{thm:maslov-0}]  As mentioned above, the proof 
  follows the same steps as the proof in \cite[Section~4]{SabVelWon21:MaxLeg}.  
  We reduce each of $\leg$ and $\leg'$ to Maslov-$0$ handle graphs on 
  (stabilized) unknots using the following steps:

\begin{description}
\item[Elimination of negative crossings]  A negative crossing is eliminated by 
  a single pinch move as in \cite[Figure~17]{SabVelWon21:MaxLeg}; we may assume 
  the pinch is ordered by choosing to pinch on the appropriate side of the 
  crossing.  Use \fullref{lem:maslov-0-surgery} to replace the cobordism 
  arising from the pinch with a Maslov-$0$ cobordism whose negative end has the 
  same crossings as the original negative end.

\item[Elimination of positive crossings] A leftmost positive crossing is 
  eliminated by the more elaborate procedure depicted in 
  \cite[Figure~18]{SabVelWon21:MaxLeg}.  Only the last step needs to be 
  adjusted, first by raising the Maslov potential of the top strand of the 
  final two pinches using the same combination of Reidemeister-I moves and 
  pinches in earlier stages of the procedure, and then by replacing the final 
  two pinches using \fullref{lem:maslov-0-surgery}.

\item[Merging of components] This step may be adjusted by simply adding 
  appropriate Reidemeister-I moves to the procedure in 
  \cite[Figure~19]{SabVelWon21:MaxLeg} to ensure that the pinches are all 
  Maslov-$0$ operations.
\end{description}

The unknots are then stabilized further using \fullref{lem:maslov-0-stab} until 
they have the same Thurston--Bennequin numbers (they already have rotation 
number $0$ by construction).  The resulting unknots are Legendrian isotopic; we 
call the result $\leg_-$.  By dragging the handle graphs along with the isotopy 
and taking a union of the handles, we obtain $(\graph_-,\leg_-)$. By 
\fullref{lem:flat-handle-lagr}, the partial surgery procedure in 
\cite{SabVelWon21:MaxLeg} results in the desired Maslov-$0$ Lagrangian 
cobordisms.
\end{proof}


%
%


%% file: sec_lqc-e.tex


\section{Relationship to non-classical invariants}
\label{sec:non-classical}

In this section, we discuss the relationship between non-classical invariants 
and Lagrangian zigzag cobordism.  The section begins with a 
review of Legendrian contact homology in \fullref{ssec:lch-background}, 
proceeds to apply it to zigzag cobordisms in \fullref{ssec:lch-zz}, and ends 
with a proof of a setwise Poincar\'e--Chekanov polynomial geography theorem 
(\fullref{thm:poly-geography}).

\subsection{Background on Legendrian contact homology}
\label{ssec:lch-background}

Lagrangian cobordisms are obstructed by the Legendrian contact homology 
differential graded algebra (LCH DGA), whose homology is a non-classical 
invariant of Legendrian knots.  In this subsection, we briefly set notation and 
review key results.  See the survey \cite{EtnNg22:LCHSurvey} for a more 
comprehensive introduction.

The LCH DGA $(\alg_{\leg},\partial_{\leg})$ of the Legendrian knot $\leg$ in 
the standard contact $\R^3$ was first defined by Chekanov \cite{Che02:DGA}; see 
also \cite{Eli98:ContactInv}. The algebra (over the field $\FF_2$, for 
simplicity) is generated by the Reeb chords of $\leg$ and graded by a 
Conley--Zehnder index.  The differential counts certain immersed disks in 
$\R^2$ with boundary on the Lagrangian projection of $\leg$; these disks 
correspond to holomorphic disks in the symplectization $\R \times \R^3$ 
\cite{EliGivHof00:SFTIntro, EtnNgSab02:LCHHol}.  

While it can be difficult to distinguish knots from presentations of their DGAs, 
Chekanov's linearization procedure yields a more computable set of invariants.  
Given a DGA $(\alg_\leg, \partial_\leg)$, its \dfn{augmentations} are DGA maps 
$\aug \colon (A_\leg, \partial_\leg)\to(\FF_2,0)$. Each augmentation $\aug$ 
induces a differential $\partial^\aug_\leg$ on the $\FF_2$ vector space 
$A_\leg$ generated by the Reeb chords of $\leg$.  The homology of the chain 
complex $(A_\leg,\partial^\aug_\leg)$ is called the \dfn{linearized Legendrian 
  contact homology $\LCH_*(\leg, \aug)$ with respect to $\aug$}.  It is 
convenient to record the dimensions of $\LCH_*(\leg, \aug)$ in a 
\dfn{Poincar\'e--Chekanov polynomial} \[p_{\leg, \aug}(t) = 
  \sum_{i=-\infty}^\infty \dim \LCH_i(\leg, \aug)\,t^i.\]
A duality result for Legendrian contact homology \cite{Sab06:LCHDuality} 
implies that every Poincar\'e--Chekanov polynomial is of the form 
\begin{equation}
  \label{eq:duality}
  p_{\leg, \aug}(t) = t + q(t) + q(t^{-1})
\end{equation} for some polynomial $q$. The set $\poly_\leg$ of 
Poincar\'e--Chekanov polynomials over all possible augmentations is an 
invariant of the Legendrian knot
$\leg$.  For example, Chekanov used this fact to show that the twist knots 
$\leg$ and $\leg'$ in \fullref{fig:chv-ex} are not Legendrian isotopic despite 
having the same classical invariants since $\poly_\leg = \{2+t\}$ while 
$\poly_{\leg'} = \{t^{-2} + t + t^2\}$ \cite{Che02:DGA}.
\begin{figure}[!htb]
 \includegraphics{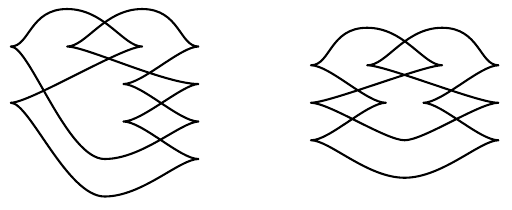}
 \caption{Two non-isotopic Legendrian representatives of $m (5_2)$.}
 \label{fig:chv-ex}
\end{figure}

As hinted above, Legendrian contact homology is functorial under Lagrangian 
cobordism.  More precisely, an exact, Maslov-$0$ Lagrangian cobordism $L$ from 
$\leg_-$ to $\leg_+$ induces a DGA map $\phi_L: 
(\alg_{\leg_+},\partial_{\leg_+}) \to (\alg_{\leg_-},\partial_{\leg_-})$ 
\cite{EkhHonKal16:LagCob}.  Such a DGA map can be used to pull back 
augmentations for $\leg_-$ to augmentations for $\leg_+$ via $\phi^*_L \aug_- = 
\aug_- \circ \phi_L$, and hence linearized Legendrian contact homology can be 
used to obstruct Lagrangian cobordisms.  We recall this obstruction from Pan 
\cite{Pan17:LagCobAug}:

\begin{theorem}[{\cite[Corollary~1.4]{Pan17:LagCobAug}}]
  \label{thm:lch-obstruction}
If there is an exact, Maslov-$0$ Lagrangian cobordism $L$ from $\leg_-$ to $\leg_+$, and if $\epsilon_-$ is an augmentation of $(\alg_{\leg_-},\partial_{\leg_-})$, then
		\[
      \LCH_*(\leg_+,\phi_L^*\aug_-) \cong 
      \LCH_*(\leg_-,\aug_-)\oplus\FF^{-\chi(L)}[0],\]
		where $\FF^{-\chi(L)}[0]$ denotes the vector space $\FF^{-\chi(L)}$ in degree 0.  In particular, if $L$ is a concordance, then the associated linearized Legendrian contact homologies are isomorphic. 
\end{theorem}

\subsection{Linearized LCH and zigzag cobordisms}
\label{ssec:lch-zz}

\fullref{thm:lch-obstruction} shows that if there is a Lagrangian concordance 
$L$ from $\leg_-$ to $\leg_+$, then $\poly_{\leg_-} \subset \poly_{\leg_+}$.  
This shows, for example, that not only are the twist knots in 
\fullref{fig:chv-ex} not Legendrian isotopic, but there is no Lagrangian 
concordance between them in either direction.

Even so, the structure of the linearized LCH invariant does not yield an 
obstruction to Maslov-$0$ zigzag cobordism.  To see why, suppose that 
$\legjoin(\leg_1,\leg_2)$ is a Lagrangian cospan with top Legendrian $\leg_+$.  
While the sets of augmentations of $\leg_i$, $i=1,2$, both pull back to 
augmentations of $\leg_+$, these sets may be disjoint.  In particular, we might 
have that the subsets $\poly_{\leg_1}, \poly_{\leg_2} \subset \poly_{\leg_+}$ 
are disjoint, and hence we cannot fruitfully use comparisons between the sets 
$\poly_{\leg_i}$ to obstruct a Lagrangian zigzag cobordism.

\begin{remark}
  This issue of not being able to guarantee that the sets of pulled-back 
  invariants match up---and hence being unable to use the sets of invariants to 
  obstruct Lagrangian zigzag cobordisms---is common to linearized LCH, 
  generating family homology \cite{SabTra13:LagCobObstructions, 
    Tra01:LegGenFunc}, and sheaf invariants \cite{Li22:LagCobSheaves, 
    SheTreZas17:LegKnotsSheaves}.
  
  The use of Legendrian invariants $\lossh (\leg)$ in knot Heegaard Floer 
  homology \cite{BalLidWon22:LagCobHFK,GolJuh19:LOSSConc} or knot monopole 
  Floer homology \cite{BalSiv18:KHMLeg} suffers from an issue that is perhaps 
  similar in spirit, but differently manifested. In, say, Heegaard Floer 
  theory, one has homomorphisms $\mathcal{F}_i \colon \HFKh (\leg_+) \to \HFKh 
  (\leg_i)$ such that $\mathcal{F}_i (\lossh (\leg_+)) = \lossh (\leg_i)$.  But
  knowing only $\lossh (\leg_i)$---even if one vanishes and the other does 
  not---one is never able to disprove the existence of \emph{some} unknown, 
  non-vanishing $\lossh (\leg_+)$ and corresponding homomorphisms 
  $\mathcal{F}_i$ satisfying the above, which would be necessary to obstruct 
  all possible zigzag cobordisms between $\leg_1$ and $\leg_2$.
\end{remark}

\begin{openquestion}
  \label{openq:leg-botany}
\fullref{fig:chv-ex} showcases an instance of the Legendrian botany question:
How many different Legendrian knots (up to isotopy) are there for a fixed set 
of classical invariants?
Since the classical invariants are preserved by Lagrangian zigzag concordance, 
it is reasonable to ask whether the two knots in \fullref{fig:chv-ex} are 
zigzag concordant.  Pushing further, one could restrict $\LagGenus$ to the set 
of Legendrians with common smooth knot class and classical invariants to 
\emph{quantify} the botany question: What is the diameter of such a set of 
Legendrian knots? An effective nonclassical invariant of Lagrangian zigzag 
concordance seems necessary to answer this question.
\end{openquestion}

One possible way to deal with the structural failure of linearized LCH to yield 
an obstruction to zigzag cobordism is as follows:
One could refine the Lagrangian zigzag cobordism relation to take pairs $(\leg, 
\aug)$ of Legendrians and augmentations  as its objects, while insisting on 
matching pullbacks of augmentations at the top in a Lagrangian cospan.  More 
precisely, an \dfn{augmented Lagrangian cospan} $\legjoin(\aug,\aug')$ consists 
of an augmented Legendrian link $(\leg_+, \aug_+)$ and two exact, connected 
Maslov-$0$ Lagrangian cobordisms $L$ from $\leg$ to $\leg_+$ and $L'$ from  
$\leg'$ to $\leg_+$ so that $\aug \circ \phi_L = \aug_+ = \aug' \circ 
\phi_{L'}$; the definitions of the augmented zigzag cobordism relation $\qcobe$ 
and concordance relation $\qconce$ then follow the same structure as 
\fullref{defn:lqc-relation}.  Note that \fullref{thm:lch-obstruction} implies 
that linearized LCH is an invariant of augmented zigzag concordance and that 
components of linearized LCH in nonzero grading are invariants of augmented 
zigzag cobordism. This refinement shifts the study of Lagrangian zigzag 
concordance from Legendrians to augmentations.  For example, we prove the 
following:

\begin{proposition}
	For any $n \in \N$, there exists a Legendrian $\leg_n$ with augmentations $\aug_1, \ldots \aug_n$ so that $(\leg_n,\aug_i) \not \qcobe (\leg_n,\aug_j)$ for $i \neq j$.
\end{proposition}

\begin{proof}
  Following the proof of \cite[Corollary~5.2]{Siv11:BorderedDGA}, let $\leg_0$ 
  be the maximal Legendrian unknot, and let $\leg_n$ be the Legendrian (i.e.\ 
  $\tb$-twisted) Whitehead double of $\leg_{n-1}$.  Sivek showed that $\leg_n$ 
  has augmentations $\{\aug_1, \ldots, \aug_n\}$ with Poincar\'e-Chekanov 
  polynomials $\{p_1, \ldots, p_n\}$ that satisfy $p_1(t) = t+2$ and
  \begin{equation} \label{eq:whitehead-recur}
    p_k(t) = t+2 + (t+2+t^{-1})(p_{k-1}(t)-t).
  \end{equation}

  Using \eqref{eq:whitehead-recur}, we may inductively compute that, for $n 
  \geq 2$, $\deg p_n = n-1$.  The proof now follows from 
  \fullref{thm:lch-obstruction}, as noted above.
\end{proof}

In a different direction, we may also \emph{exploit} the interaction between 
lineared LCH and a Lagrangian cospan, not to obstruct zigzag cobordisms, but to 
prove other interesting statements.
In particular, we may use
the fact that the Legendrian $\leg_+$ at the top of a Lagrangian cospan 
$\legjoin(\leg_1,\leg_2)$ pulls back linearized LCH information from both 
$\leg_1$ and $\leg_2$
to
construct Legendrian knots with (almost) arbitrary sets of Poincar\'e--Chekanov 
polynomials.  The procedure is made precise by \fullref{thm:poly-geography}, 
which tells us that for any set of Laurent polynomials compatible with duality 
as in \eqref{eq:duality}, there is a Legendrian knot that contains those 
polynomials---up to a correction in grading $0$---in its set of 
Poincar\'e--Chekanov polynomials.

\begin{proof}[Proof of \fullref{thm:poly-geography}]
  By \cite[Theorem 1.2]{MelShr05:ChePolys}, for each $i \in \{1, \ldots, n\}$, 
  there exists a Legendrian knot $\leg_i$ and an augmentation $\aug_i$ of its 
  LCH DGA so that $p_{\leg_i, \aug_i}(t) = p_i(t)+p_i(t^{-1})+t$.  
  \fullref{thm:maslov-0} and induction yield a Legendrian knot $\leg_+$ and 
  Maslov-$0$ Lagrangian cobordisms $\leg_i \prec_{L_i} \leg_+$. Let $c_i = 
  \chi(L_i)$. Finally, \fullref{thm:lch-obstruction} gives augmentations 
  $\aug_{+,i}$ of $\leg_+$ so that \[p_{\leg_+, \aug_{+,i}}(t) = p_{\leg_i, 
      \aug_i}(t) + \chi(L_i) =  p_i(t)+p_i(t^{-1})+t+ c_i,\]
which proves the theorem.
\end{proof}

\begin{remark}
  The constants $c_i$ are not independent. In fact, once $c_1$ is known, the 
  other constants are determined.  The idea is that the formula $\tb(\leg_+) = 
  p_{\leg_i, \aug_i}(-1) + c_i$ holds for any $i$.  Knowing $c_1$ determines 
  $\tb(\leg_+)$, and then, in turn, $\tb(\leg_+)$ determines $c_2, \ldots, 
  c_n$.
\end{remark}

%% file: lqc.bbl
\providecommand{\bysame}{\leavevmode\hbox to3em{\hrulefill}\thinspace}
\providecommand{\MR}{\relax\ifhmode\unskip\space\fi MR }
\providecommand{\MRhref}[2]{%
  \href{http://www.ams.org/mathscinet-getitem?mr=#1}{#2}
}
\providecommand{\href}[2]{#2}
\begin{thebibliography}{BLW22}

\bibitem[AB20]{AnBae20:LegGraphsDGA}
Byung~Hee An and Youngjin Bae, \emph{A {C}hekanov-{E}liashberg algebra for
  {L}egendrian graphs}, J. Topol. \textbf{13} (2020), no.~2, 777--869.
  \MR{4092780}

\bibitem[ABK22]{AnBaeKal22:LegGraphsDGA}
Byung~Hee An, Youngjin Bae, and Tam\'{a}s K\'{a}lm\'{a}n, \emph{Ruling
  invariants for {L}egendrian graphs}, J. Symplectic Geom. \textbf{20} (2022),
  no.~1, 49--97. \MR{4518248}

\bibitem[Ago22]{Ago22:RibConcPartialOrder}
Ian Agol, \emph{Ribbon concordance of knots is a partial ordering}, Comm. Amer.
  Math. Soc. \textbf{2} (2022), 374--379. \MR{4520779}

\bibitem[BG22]{BouGal22:GeoBilinLCH}
Fr{\'e}d{\'e}ric Bourgeois and Damien Galant, \emph{Geography of bilinearized
  legendrian contact homology}, version 4, 2022,
  \href{http://arxiv.org/abs/1905.12037}{\texttt{arXiv:1905.12037}}.

\bibitem[BLW22]{BalLidWon22:LagCobHFK}
John~A. Baldwin, Tye Lidman, and C.-M.~Michael Wong, \emph{Lagrangian
  cobordisms and {L}egendrian invariants in knot {F}loer homology}, Michigan
  Math. J. \textbf{71} (2022), no.~1, 145--175. \MR{4389674}

\bibitem[BS18]{BalSiv18:KHMLeg}
John~A. Baldwin and Steven Sivek, \emph{Invariants of {L}egendrian and
  transverse knots in monopole knot homology}, J. Symplectic Geom. \textbf{16}
  (2018), no.~4, 959--1000. \MR{3917725}

\bibitem[BST15]{BouSabTra15:LagCobGF}
Fr\'{e}d\'{e}ric Bourgeois, Joshua~M. Sabloff, and Lisa Traynor,
  \emph{Lagrangian cobordisms via generating families: construction and
  geography}, Algebr. Geom. Topol. \textbf{15} (2015), no.~4, 2439--2477.
  \MR{3402346}

\bibitem[CG22]{CasGao22:InfLag}
Roger Casals and Honghao Gao, \emph{Infinitely many {L}agrangian fillings},
  Ann. of Math. (2) \textbf{195} (2022), no.~1, 207--249. \MR{4358415}

\bibitem[CH18]{CocHar18:ConcGeom}
Tim Cochran and Shelly Harvey, \emph{The geometry of the knot concordance
  space}, Algebr. Geom. Topol. \textbf{18} (2018), no.~5, 2509--2540.
  \MR{3848393}

\bibitem[Cha10]{Cha10:LagConc}
Baptiste Chantraine, \emph{Lagrangian concordance of {L}egendrian knots},
  Algebr. Geom. Topol. \textbf{10} (2010), no.~1, 63--85. \MR{2580429}

\bibitem[Cha15]{Cha15:LagConcNotSym}
\bysame, \emph{Lagrangian concordance is not a symmetric relation}, Quantum
  Topol. \textbf{6} (2015), no.~3, 451--474. \MR{3392961}

\bibitem[Che02]{Che02:DGA}
Yuri Chekanov, \emph{Differential algebra of {L}egendrian links}, Invent. Math.
  \textbf{150} (2002), no.~3, 441--483. \MR{1946550}

\bibitem[CN22]{CasNg22:InfLag}
Roger Casals and Lenhard Ng, \emph{Braid loops with infinite monodromy on the
  {L}egendrian contact {DGA}}, preprint, version 2, 2022,
  \href{http://arxiv.org/abs/2101.02318}{\texttt{arXiv:2101.02318}}.

\bibitem[CNS16]{CorNgSiv16:LagConcObstructions}
Christopher Cornwell, Lenhard Ng, and Steven Sivek, \emph{Obstructions to
  {L}agrangian concordance}, Algebr. Geom. Topol. \textbf{16} (2016), no.~2,
  797--824. \MR{3493408}

\bibitem[Dim16]{Dim16:LegAmbSurg}
Georgios Dimitroglou~Rizell, \emph{Legendrian ambient surgery and {L}egendrian
  contact homology}, J. Symplectic Geom. \textbf{14} (2016), no.~3, 811--901.
  \MR{3548486}

\bibitem[DP13]{DynPra13:MinGridMaxTB}
I.~A. Dynnikov and M.~V. Prasolov, \emph{Bypasses for rectangular diagrams. {A}
  proof of the {J}ones conjecture and related questions}, Trans. Moscow Math.
  Soc. \textbf{74} (2013), 97--144. \MR{3235791}

\bibitem[EGH00]{EliGivHof00:SFTIntro}
Y.~Eliashberg, A.~Givental, and H.~Hofer, \emph{Introduction to symplectic
  field theory}, Visions in mathematics: {T}owards 2000, Geom. Funct. Anal.
  {\bf {S}pecial {V}ol. {II}}, Birkh\"{a}user Verlag, Basel, 2000, Proceedings
  of the meeting held at {T}el {A}viv {U}niversity, {T}el {A}viv, {A}ugust
  25--{S}eptember 3, 1999, pp.~560--673. \MR{1826267}

\bibitem[EH03]{EtnHon03:LegConnSum}
John~B. Etnyre and Ko~Honda, \emph{On connected sums and {L}egendrian knots},
  Adv. Math. \textbf{179} (2003), no.~1, 59--74. \MR{2004728}

\bibitem[EHK16]{EkhHonKal16:LagCob}
Tobias Ekholm, Ko~Honda, and Tam\'{a}s K\'{a}lm\'{a}n, \emph{Legendrian knots
  and exact {L}agrangian cobordisms}, J. Eur. Math. Soc. (JEMS) \textbf{18}
  (2016), no.~11, 2627--2689. \MR{3562353}

\bibitem[Eli98]{Eli98:ContactInv}
Yakov Eliashberg, \emph{Invariants in contact topology}, Proceedings of the
  {I}nternational {C}ongress of {M}athematicians, {D}oc. {M}ath. {\bf {E}xtra
  {V}ol. {II}}, Deutsche Math. Ver., 1998, pp.~327--338. \MR{1648083}

\bibitem[EN22]{EtnNg22:LCHSurvey}
John~B. Etnyre and Lenhard~L. Ng, \emph{Legendrian contact homology in
  {$\mathbb{R}^3$}}, Surveys in 3-manifold topology and geometry, Surv. Differ.
  Geom., vol.~25, Int. Press, Boston, MA, 2020 (2022), pp.~103--161.
  \MR{4479751}

\bibitem[ENS02]{EtnNgSab02:LCHHol}
John~B. Etnyre, Lenhard~L. Ng, and Joshua~M. Sabloff, \emph{Invariants of
  {L}egendrian knots and coherent orientations}, J. Symplectic Geom. \textbf{1}
  (2002), no.~2, 321--367. \MR{1959585}

\bibitem[Etn05]{Etn05:LegTransSurvey}
John~B. Etnyre, \emph{Legendrian and transversal knots}, Handbook of knot
  theory, Elsevier B. V., Amsterdam, 2005, pp.~105--185. \MR{2179261}

\bibitem[EV18]{EtnVer18:LegSatellites}
John Etnyre and Vera V\'{e}rtesi, \emph{Legendrian satellites}, Int. Math. Res.
  Not. IMRN (2018), no.~23, 7241--7304. \MR{3883132}

\bibitem[FT97]{FucTab97:LegTransKnots}
Dmitry Fuchs and Serge Tabachnikov, \emph{Invariants of {L}egendrian and
  transverse knots in the standard contact space}, Topology \textbf{36} (1997),
  no.~5, 1025--1053. \MR{1445553}

\bibitem[Gei08]{Gei08:ContactBook}
Hansj\"{o}rg Geiges, \emph{An introduction to contact topology}, Cambridge
  Studies in Advanced Mathematics, vol. 109, Cambridge University Press,
  Cambridge, 2008. \MR{2397738}

\bibitem[GJ19]{GolJuh19:LOSSConc}
Marco Golla and Andr\'{a}s Juh\'{a}sz, \emph{Functoriality of the {EH} class
  and the {LOSS} invariant under {L}agrangian concordances}, Algebr. Geom.
  Topol. \textbf{19} (2019), no.~7, 3683--3699. \MR{4045364}

\bibitem[GSY22]{GuaSabYac22:LegSatellitesCob}
Roberta Guadagni, Joshua~M. Sabloff, and Matthew Yacavone, \emph{Legendrian
  satellites and decomposable cobordisms}, J. Knot Theory Ramifications
  \textbf{31} (2022), no.~13, Paper No. 2250071, 33. \MR{4523297}

\bibitem[K{\'a}l05]{Kal05:LCHpi1}
Tam\'{a}s K{\'a}lm\'{a}n, \emph{Contact homology and one parameter families of
  {L}egendrian knots}, Geom. Topol. \textbf{9} (2005), 2013--2078. \MR{2209366}

\bibitem[Li22]{Li22:LagCobSheaves}
Wenyuan Li, \emph{Lagrangian cobordism functor in microlocal sheaf theory},
  preprint, version 2, 2022,
  \href{http://arxiv.org/abs/2108.10914}{\texttt{arXiv:2108.10914}}.

\bibitem[LM]{LivMoo:KnotInfo}
Charles Livingston and Allison~H. Moore, \emph{Knot{I}nfo: {T}able of {K}not
  {I}nvariants}, accessed online on Jul 15, 2023.

\bibitem[MS05]{MelShr05:ChePolys}
Paul Melvin and Sumana Shrestha, \emph{The nonuniqueness of {C}hekanov
  polynomials of {L}egendrian knots}, Geom. Topol. \textbf{9} (2005),
  1221--1252. \MR{2174265}

\bibitem[Ng01]{Ng01:LegSatellites}
Lenhard~L. Ng, \emph{The {L}egendrian satellite construction}, preprint,
  version 1, 2001,
  \href{http://arxiv.org/abs/math/0112105}{\texttt{arXiv:math/0112105}}.

\bibitem[NT04]{NgTra04:LegTorusLinks}
Lenhard Ng and Lisa Traynor, \emph{Legendrian solid-torus links}, J. Symplectic
  Geom. \textbf{2} (2004), no.~3, 411--443. \MR{2131643}

\bibitem[OP12]{ODoPav12:LegGraphs}
Danielle O'Donnol and Elena Pavelescu, \emph{On {L}egendrian graphs}, Algebr.
  Geom. Topol. \textbf{12} (2012), no.~3, 1273--1299. \MR{2966686}

\bibitem[Pan17]{Pan17:LagCobAug}
Yu~Pan, \emph{The augmentation category map induced by exact {L}agrangian
  cobordisms}, Algebr. Geom. Topol. \textbf{17} (2017), no.~3, 1813--1870.
  \MR{3677941}

\bibitem[RS20]{RutSul20:CellularLCH1}
Dan Rutherford and Michael Sullivan, \emph{Cellular {L}egendrian contact
  homology for surfaces, part {I}}, Adv. Math. \textbf{374} (2020), 107348, 71.
  \MR{4133520}

\bibitem[Sab06]{Sab06:LCHDuality}
Joshua~M. Sabloff, \emph{Duality for {L}egendrian contact homology}, Geom.
  Topol. \textbf{10} (2006), 2351--2381. \MR{2284060}

\bibitem[Sab21]{Sab21:RulingsIntro}
\bysame, \emph{Ruling and augmentation invariants of {L}egendrian knots},
  Encyclopedia of knot theory, CRC Press, Boca Raton, FL, 2021, pp.~403--410.

\bibitem[Sar20]{Sar20:RibKh}
Sucharit Sarkar, \emph{Ribbon distance and {K}hovanov homology}, Algebr. Geom.
  Topol. \textbf{20} (2020), no.~2, 1041--1058. \MR{4092319}

\bibitem[Siv11]{Siv11:BorderedDGA}
Steven Sivek, \emph{A bordered {C}hekanov-{E}liashberg algebra}, J. Topol.
  \textbf{4} (2011), no.~1, 73--104. \MR{2783378}

\bibitem[ST13]{SabTra13:LagCobObstructions}
Joshua~M. Sabloff and Lisa Traynor, \emph{Obstructions to {L}agrangian
  cobordisms between {L}egendrians via generating families}, Algebr. Geom.
  Topol. \textbf{13} (2013), no.~5, 2733--2797. \MR{3116302}

\bibitem[STZ17]{SheTreZas17:LegKnotsSheaves}
Vivek Shende, David Treumann, and Eric Zaslow, \emph{Legendrian knots and
  constructible sheaves}, Invent. Math. \textbf{207} (2017), no.~3, 1031--1133.
  \MR{3608288}

\bibitem[SVW21]{SabVelWon21:MaxLeg}
Joshua~M. Sabloff, David~Shea Vela-Vick, and C.-M.~Michael Wong, \emph{Upper
  bounds for the {L}agrangian cobordism relation on {L}egendrian links},
  {A}lgebr.\ {G}eom.\ {T}opol., to appear, version 1, 2021,
  \href{http://arxiv.org/abs/2105.02390}{\texttt{arXiv:2105.02390}}.

\bibitem[Tra01]{Tra01:LegGenFunc}
Lisa Traynor, \emph{Generating function polynomials for {L}egendrian links},
  Geom. Topol. \textbf{5} (2001), 719--760. \MR{1871403}

\bibitem[Tra21]{Tra21:LegTransIntro}
\bysame, \emph{An introduction to the world of {L}egendrian and transverse
  knots}, Encyclopedia of knot theory, CRC Press, Boca Raton, FL, 2021,
  pp.~385--392.

\end{thebibliography}
